\documentclass[16pt]{article}
\usepackage{amsmath,amsthm,amsfonts,amssymb}
\usepackage{amssymb,amsmath,graphics}
 \usepackage{epsfig,psfrag}
\usepackage{graphicx}
\usepackage{psfig}
 \usepackage{verbatim}

\newtheorem{theorem}{Theorem}[section]
\newtheorem{lemma}[theorem]{Lemma}
\newtheorem{proposition}[theorem]{Proposition}
\newtheorem{definition}[theorem]{Definition}

\newtheorem{remark}[theorem]{Remark}

\newcommand\RR{{\Bbb R}}
\newcommand\CC{{\Bbb C}}

\newcommand\ZZ{{\Bbb Z}}
\newcommand\TT{{\Bbb T}}

\newcommand\supp{{\rm supp}}

\begin{document}
\title{Continuous Wavelets on Compact Manifolds}
\author{Daryl Geller\\
\footnotesize\texttt{{Department of Mathematics, Stony Brook University, Stony Brook, NY 11794-3651}}\\
\footnotesize\texttt{{daryl@math.sunysb.edu}}\\
 \thanks{This work was partially  supported by the Marie Curie Excellence Team Grant MEXT-CT-2004-013477, Acronym MAMEBIA.}
Azita Mayeli \\
\footnotesize\texttt{{Department of Mathematics, Stony Brook University, Stony Brook, NY 11794-3651}}\\
\footnotesize\texttt{{amayeli@math.sunysb.edu}}}

\maketitle

\begin{abstract}
Let $\bf M$ be a smooth compact oriented Riemannian manifold,
and let $\Delta_{\bf M}$ be the Laplace-Beltrami operator on ${\bf M}$. Say $0 \neq f \in
\mathcal{S}(\RR^+)$, and that $f(0) = 0$.  For $t > 0$, let $K_t(x,y)$ denote
the kernel of $f(t^2 \Delta_{\bf M})$. We show that $K_t$ is well-localized near the diagonal,
in the sense that it satisfies estimates akin to those satisfied by the kernel of the convolution
operator $f(t^2\Delta)$ on $\RR^n$.  We define
continuous ${\cal S}$-wavelets on ${\bf M}$, in such a manner that
$K_t(x,y)$ satisfies this definition, because of its localization near the diagonal.
Continuous ${\cal S}$-wavelets on ${\bf M}$ are analogous to continuous wavelets
on $\RR^n$ in $\mathcal{S}(\RR^n)$.  In particular, we are able to characterize
the H$\ddot{o}$lder continuous functions on ${\bf M}$ by the size of their continuous
${\mathcal{ S}}-$wavelet transforms,
for H$\ddot{o}$lder exponents strictly between $0$ and $1$.
If $\bf M$ is the torus $\TT^2$ or the sphere $S^2$, and $f(s)=se^{-s}$ (the ``Mexican hat'' situation),
we obtain two explicit approximate formulas
for $K_t$, one to be used when $t$ is large, and one to be used when $t$ is small.\\

\footnotesize{
 \begin{tabular}{lrl}
 Keywords and phrases:  &   \multicolumn{2}{l} {\em Frames, Wavelets, Continuous Wavelets, Spectral Theory, Schwartz Functions,}\\
  &  \hspace{-.1cm}  {\em  Time-Frequency  Analysis,  Manifolds, Sphere, Torus, Pseudodifferential Operators,}\\
 &  \hspace{-1cm}  {\em  H$\ddot{o}$der Spaces.}
  \end{tabular}
 \vspace{.3cm}

 \begin{tabular}{lrl}
  &&  \hspace{-1cm} AMS  Classification;  {42C40, 42B20, 58J40, 58J35, 35P05.}
\end{tabular}
}
 \end{abstract}

{\bf Table of Contents}
\begin{itemize}
\item
 Section \ref{sec:Introduction}: Introduction and Historical Comments
\item
 Section \ref{applying-the-spectral-theorem}: Applying the Spectral Theorem
\item
 Section \ref{preliminaries-on-manifolds}: Preliminaries on Manifolds
\item
 Section \ref{kernels}: Kernels
\item
 Section \ref{continuous-s-wavelets-on-manifolds}: Continuous ${\cal S}$-Wavelets  on Manifolds
\item
 Section \ref{homogeneous-manifolds}: Homogeneous Manifolds
 \item Section \ref{a-technical-lemma}: A Technical Lemma
\end{itemize}

\section{\large{Introduction}}
\label{sec:Introduction}

Let ${\mathcal S}(\RR^+)$ denote the space of restrictions to $\RR^+$ of functions in ${\mathcal S}(\RR)$.
Say $0 \not\equiv f_0 \in {\mathcal S}(\RR^+)$, and let
\[f(s) = sf_0(s).\]
One then has
the {\em Calder\'on formula} for $f$: if $c \in (0,\infty)$ is defined by
\[c = \int_0^{\infty} |f(t)|^2 \frac{dt}{t} = \int_0^{\infty} t|f_0(t)|^2 dt,\]
then for all $s > 0$,

\begin{equation}
\label{cald}
\int_0^{\infty} |f(ts)|^2 \frac{dt}{t} = c < \infty.
\end{equation}

The even function $h(\xi) = f(\xi^2)$ is then in ${\mathcal S}(\RR)$ and satisfies

\begin{equation}
\label{cald1}
\int_0^{\infty} |h(t\xi)|^2 \frac{dt}{t} = \frac{c}{2} < \infty.
\end{equation}

(In fact, all even functions in
${\mathcal S}(\RR)$ satisfying (\ref{cald1}) arise in this manner).

Its inverse Fourier transform $\psi = \check{h}$ is admissible (i.e. is a continuous wavelet).  That is,
for some $c' > 0$,
\begin{equation}
\label{ctswvrn0}
\int_0^{\infty} \|F*\psi_t\|^2_2 \frac{dt}{t} = c' \|F\|^2_2
\end{equation}
for all $F \in L^2(\RR^n)$.  Here, as usual, $\psi_t(x) = t^{-1} \psi(x/t)$.\\

 We prefer to write, formally,
\[ \check{h} = f(-d^2/dx^2) \delta; \]
the formal justfication being that

\[ (f(-d^2/dx^2) \delta)\hat{\:} = f(\xi^2) = h(\xi). \]
Thus $f(-d^2/dx^2) \delta$ is a continuous wavelet on $\RR$.

Our program is to construct analogues of continuous wavelets, on
much more general spaces, by replacing the positive number $s$ in (\ref{cald})
by a positive self-adjoint operator $T$ on a Hilbert space ${\cal H}$.
If $P$ is the projection onto the null space of $T$, by the spectral theorem
we obtain the relation

\begin{equation}
\label{gelmay1}
\int_0^{\infty} |f|^2(tT) \frac{dt}{t} = c(I-P),
\end{equation}

where the integral in (\ref{gelmay1}) converges strongly.  (\ref{gelmay1}) will be justified in the next section.

Taking $T$ to be $-d^2/dx^2$ on $\RR$ leads to the continuous wavelet $f(-d^2/dx^2) \delta$ on
$\RR$. (Of course, on $\RR$, $P = 0$.)

We began our program of looking at more general positive self-adjoint operators $T$,
in order to construct continuous wavelets, in our article \cite{gm1}.  There we took $T$
to be the sublaplacian $L$ on $L^2(G)$, where $G$ is a stratifed Lie
group, and thereby obtained continuous wavelets and frames on such $G$.
In fact, in that context, $f(L)\delta$ was a continuous wavelet. (Again, in that
context, $P$ was zero.)  (Our article \cite{gm1} was motivated by the second author's
thesis \cite{Mayelithesis05}, where it was shown that if $f(s) = se^{-s}$, then the
``Mexican hat'' $f(L)\delta$ is a continuous wavelet on the Heisenberg group.)

In this article we will look at the (much more practical!)
situation in which $T$ is the Laplace-Beltrami operator on $L^2({\bf M})$,
where ${\bf M}$ is a smooth compact oriented Riemannian
manifold without boundary.  We will construct analogues of continuous Schwartz wavelets in
this context, and will obtain explicit formulas for them if
${\bf M}$ is the sphere or the torus.  In a sequel article (\cite{gmfr}) we shall use a discrete analogue of (\ref{gelmay1}) to
construct nearly tight frames in
this situation, and show that they are appropriately well-localized (specifically, that
they satisfy a space-frequency analysis).  In both of these articles, $P$ will be the projection onto the
one-dimensional space of constant functions.

We now summarize our results and methods.  We discuss prior work on wavelets on manifolds, especially
the work of Narcowich, Petrushev and Ward on the sphere \cite{narc1}, \cite{narc2}, at the end of this
introduction.
\\

In the model situation (\ref{cald}) -- (\ref{ctswvrn0}) on $\RR$,
note that the kernel of the operator $F \rightarrow F*\psi_t$ is $K_t(x,y) := t^{-1} \psi((x-y)/t)$.
Since $h(0)=0$, we have
$\int K_t(x,y) dx = 0$ for all $y$ and $\int K_t(x,y) dy = 0$ for all $x$. Also $K_t(x,y)$ is smooth in $t,x,y$ for
$t > 0$.

Motivated by this model case, we define continuous wavelets on ${\bf M}$ as follows.
Suppose that the function $K_t(x,y)$ is smooth for $t > 0$, $x,y \in {\bf M}$.
For $t > 0$, define $T_t: L^2({\bf M}) \rightarrow C^{\infty}({\bf M})$ to be the operator with kernel $K_t$,
We define $K_t(x,y)$ to be a {\em continuous wavelet} on ${\bf M}$, provided that for some $c > 0$,
\begin{equation}
\label{ttfiso}
\int_0^{\infty} \|T_t F\|^2_2 \frac{dt}{t} = c \|(I-P)F\|^2_2,
\end{equation}
for all $F \in L^2({\bf M})$,
and that for any $t > 0$, $\int_{\bf M} K_t(x,y) d\mu(x) = 0$ for all $y$ and $\int_{\bf M} K_t(x,y) d\mu(y) = 0$ for all $x$.
(Here $\mu$ is the measure on ${\bf M}$ arising from integration with respect to the
volume form.)

It is then easy to see that, if $K_t$ is the kernel of $f(t^2\Delta)$ ($f$ as before),
where $\Delta = \Delta_{\bf M}$ is the Laplace-Beltrami operator on ${\bf M}$, then $K_t(x,y)$ is a
continuous wavelet on ${\bf M}$.  In particular, (\ref{ttfiso}) follows easily from (\ref{gelmay1}), simply by
applying both sides of (\ref{gelmay1}) to $F$, taking the inner product of both sides with $F$, and making the change of
variables which replaces $t$ by $t^2$ in the integral.  Note that $K_t(x,y) = [f(t^2\Delta)\delta_y](x)$,
which is analogous to $f(-d^2/dx^2) \delta$ being a continuous wavelet on $\RR$.\\

This then gives an $L^2$ theory of continuous wavelets on manifolds.  However, in practice, one
wishes to go beyond this theory, looking at continuous wavelets on other function spaces, and discretizing
them to obtain frames.  For this, one needs the wavelets to have further properties.  Fortunately,
as we shall see, if $K_t$ is the kernel of $f(t^2\Delta)$, these properties are present.

Specifically, let us return to our model situation; let us however work on $\RR^n$, not just
$\RR^1$.  Then $K_t(x,y)$, the kernel of $f(t^2\Delta)$, would be of the
form $t^{-n} \psi((x-y)/t)$ for some $\psi \in \mathcal{S}$.
(Here $\Delta$ is the usual Laplacian on $\RR^n$, and $\hat{\psi} = G$, where $G(\xi) = f(|\xi|^2)$.)
For any $N, \alpha, \beta$,
there would thus exist $C_{N, \alpha, \beta}$ such that
\begin{equation}\notag
t^{n+|\alpha|+|\beta|} \left|\left(\frac{x-y}{t}\right)^N \partial_x^{\alpha} \partial_y^{\beta} K_t(x,y)\right|
\leq C_{N, \alpha, \beta}
\end{equation}
for all $t, x, y$.  Such estimates are essential in the theory of wavelets on $\RR^n$.  We therefore
make the following definition:

\begin{definition}
\label{ctsswav}
Let $K_t(x,y)$ be a continuous wavelet on ${\bf M}$.  Then we say that $K_t(x,y)$ is a {\em continuous
${\cal S}$-wavelet} on ${\bf M}$ provided that:\\
\ \\
For every pair of
$C^{\infty}$ differential operators $X$ (in $x$) and $Y$ (in $y$) on ${\bf M}$,
and for every integer $N \geq 0$, there exists $C_{N,X,Y}$ as follows.  Suppose
$\deg X = j$ and $\deg Y = k$.  Then
\begin{equation}
\label{xykt}
t^{n+j+k} \left|\left(\frac{d(x,y)}{t}\right)^N XYK_t(x,y)\right| \leq C_{N,X,Y}
\end{equation}
for all $t > 0$ and all $x,y \in {\bf M}$.
\end{definition}

Here $d$ is the geodesic distance on ${\bf M}$.  We use the terminology ``continuous ${\cal S}$-wavelet''
to express the idea that these wavelets are analogous to Schwartz wavelets on $\RR^n$.\\
\ \\
In Lemma \ref{manmol}, we will show the following key result:\\
\ \\
{\bf Lemma \ref{manmol}} If $K_t$ is the kernel of $f(t^2\Delta)$ ($f$ as before),
then $K_t(x,y)$ satisfies (\ref{xykt}) (and hence is a continuous ${\cal S}$-wavelet on ${\bf M}$).\\
\ \\
Our proof of this lemma uses the theory of pseudodifferential operators (most crucially,
a result of Strichartz \cite{Stric72}), and Huygens' principle.

In Theorem \ref{hldchar} we shall show that
continuous ${\cal S}$-wavelets are well adapted to the study of certain other function spaces.
Specifically, we shall show the following generalization of a theorem of
Holschneider and Tchamitchian (\cite{hotch}), who worked on the real line, and
characterized the wavelet transforms of the
spaces of H\"older continuous functions (for H\"older exponent strictly between $0$ and $1$).\\
\ \\
{\bf Theorem \ref{hldchar}}\\ \textit
{Let $K_t(x,y)$ be a continuous ${\cal S}$-wavelet on ${\bf M}$, and, for $t > 0$,
let $T_t$ be the operator on $L^2$ with kernel $K_t$.
Suppose $F \in L^2({\bf M})$.  Then:\\
(a) If $F$ is H\"older continuous, with H\"older exponent $\alpha$ ($0 < \alpha \leq 1$), then for
some $C > 0$,
\begin{equation}
\label{hldcharway0}
\|T_t F\| \leq C t^{\alpha}
\end{equation}
for all $t > 0$.  (Here $\|\:\|$ denotes sup norm.)\\
(b) Conversely, say $0 < \alpha < 1$, $C > 0$, and that $F$ satisfies (\ref{hldcharway0})
for all $t > 0$.  Then $F$ is H\"older continuous, with H\"older exponent $\alpha$.}\\
\ \\
The proof of this result will be a straightforward generalization of the argument of
Holschneider and Tchamitchian (\cite{hotch}).

Our construction of nearly tight frames in our sequel article \cite{gmfr} will also
rely heavily on Lemma \ref{manmol}.
In another sequel article (already available \cite{gm3}), we will use Lemma \ref{manmol}
to show that one can determine whether
$F$ is in a Besov space, solely from a knowledge of the size of its frame coefficients.
In a future article, we hope to study the same question for Triebel-Lizorkin spaces.
(The analogous problems on $\RR^n$ were solved in \cite{fj} and \cite{fj2}.) \\

It should be noted that in (\ref{xykt}) we assume nothing about the $t$ derivatives
of $K_t(x,y)$, and no such information is needed in proving Theorem \ref{hldchar}.  However,
in fact, if $K_t$ is the kernel of $f(t^2\Delta)$, one does have the following improvement
on (\ref{xykt}):\\
\ \\
For every pair of
$C^{\infty}$ differential operators $X$ (in $x$) and $Y$ (in $y$) on ${\bf M}$,
and for all integers $m, N \geq 0$, there exists $C_{N,m,X,Y}$ as follows.  Suppose
$\deg X = j$ and $\deg Y = k$.  Then
\begin{equation}
\label{diagest3}
t^{m+n+j+k} \left|\left(\frac{d(x,y)}{t}\right)^N \left(\frac{\partial}{\partial t}\right)^m XYK_t(x,y)\right| \leq C_{N,m,X,Y}
\end{equation}
for all $t > 0$ and all $x,y \in {\bf M}$. \\

This is in fact an easy consequence of (\ref{xykt}).  For example,
$\frac{\partial}{\partial t}K_t(x,y)$ is the kernel of $2t \Delta f(t^2 \Delta)$,
and hence equals $2t \Delta_x K_t(x,y)$, which we can estimate by using (\ref{xykt}).\\

We should explain in what sense $K_t(x,y)$ (the kernel of $f(t^2\Delta)$)
deserves to be called a {\em wavelet}.  In our model situation on $\RR^n$,
$K_t(x,y) = t^{-n} \psi((x-y)/t)$ behaves in evident, well-known ways under dilation and translation:
\begin{eqnarray}
K_{rt}(rx,ry) & = & r^{-n}K_t(x,y) \mbox{ for } r > 0; \mbox{ and } \label{ktrx} \\
K_t(x+a,y+a) & = & K_t(x,y) \mbox{ for } a \in \RR^n \label{kta}
\end{eqnarray}

Now (\ref{ktrx}) says that, in the model case, $K_t(x,y)$ is homogeneous of degree $-n$ in $(t,x,y)$,
so that for any $m, \alpha, \beta$,
$t^{n+m+|\alpha|+|\beta|} \partial_t^m \partial_x^{\alpha} \partial_y^{\beta} K_t(x,y)$
is homogeneous of degree $0$, and is hence bounded as a function of $(t,x,y)$.
An analogous fact to this boundedness holds on ${\bf M}$, by (\ref{diagest3}) with $N=0$.

As for (\ref{kta}), a general manifold ${\bf M}$ has nothing
akin to translations, but in Section 5 we shall discuss
the situation in which ${\bf M}$ has a
transitive group $G$ of smooth metric isometries.  In that case one can easily see
that $K_t(Tx,Ty) = K_t(x,y)$ for all $T \in G$, which is analogous to (\ref{kta}).
(Manifolds with such a group $G$ are usually called {\em homogeneous}.
Obvious examples are the sphere and the torus.)\\
\ \\
Of course, to apply our algorithm in practice, one would need to
compute the kernels $K_t$ approximately.  We give some examples
where this can be done, in Section 5.
In the special case $f(s) = se^{-s}$, $K_t$ can be thought of,
very naturally, as a
``Mexican hat'' continuous wavelet for ${\bf M}$. (On $\RR^n$,
the Mexican hat wavelet is a multiple of $\Delta e^{-\Delta} \delta$, the
second derivative of a Gaussian, the
function whose Fourier transform is $|\xi|^2 e^{-|\xi|^2}$.)  In Section 5
we compute these continuous wavelets in the special cases where ${\bf M}$ is the
$n$-sphere $S^n$ and the $n$-torus $T^n$.   For the $2$-torus, we show that $K_t$ can
be evaluated by use of either of two different sums, one
(obtained from the eigenfunction expansion) which converges
very quickly for $t$ large, and the other which (obtained from the proof
of the Poisson summation formula) converges very quickly for $t$ small.  (The
method can be extended to general $n$.)
For the sphere, the eigenfunction expansion again gives a sum which converges
very quickly for $t$ large.  When $n=2$, for $t$ small, by use of heat trace methods,
we obtain a formula which converges very quickly, and which appears (from
numerical evidence) to be an excellent approximation to $K_t$.
Specifically, in \cite{Polt}, I. Polterovich obtains a completely explicit formula for
the heat trace asymptotics on the sphere.  (Earlier, less explicit formulae were
found earlier in (\cite{CahnWolf}) and (\cite{Camp}).)  It is not hard to see that, on
the sphere, $K_t(x,y) := h_t(x \cdot y)$ is a function of $x \cdot y$.  Using
Polterovich's result, we show how one can compute the Maclaurin series for
$4\pi h_t(\cos \theta)$.  In this manner we obtain an approximation
\begin{equation}
\label{htapp0}
4\pi h_t(\cos \theta) \sim \frac{e^{-\theta^2/4t^2}}{t^2}[(1-\frac{\theta^2}{4t^2})p(t,\theta)-t^2q(t,\theta)],
\end{equation}
where
\begin{equation}\notag
p(t,\theta)=1+\frac{t^2}{3}+\frac{t^4}{15}+\frac{4t^6}{315}+\frac{t^8}{315}+
\frac{\theta^2}{4}(\frac{1}{3}+\frac{2t^2}{15}+\frac{4t^4}{105}+\frac{4t^6}{315})
\end{equation}
and
\begin{equation}\notag
q(t,\theta)= \frac{1}{3}+\frac{2t^2}{15}+\frac{4t^4}{105}+\frac{4t^6}{315}+
\frac{\theta^2}{4}(\frac{2}{15}+\frac{8t^2}{105}+\frac{4t^4}{105})
\end{equation}
Maple says that when $t=.1$, the error in the approximation (\ref{htapp0}) is never more than $9.5 \times 10^{-4}$
for any $\theta \in [-\pi,\pi]$, even though both sides have a maximum of about 100.
(To obtain rigorous bounds on the error is research in progress, which we expect to complete soon.)
We derive a similar approximation to the heat kernel itself, when $n=2$.  The method can be
extended to general $n$. Note that, if in (\ref{htapp0})
we approximate $p \sim 1$ and $q \sim 0$, we would obtain the formula for the usual Mexican hat wavelet
on the real line, as a function of $\theta$.

\subsection{Historical comments}
A number of groups of researchers have been studying continuous wavelets and frames on manifolds, and
some have obtained important real-world applications.  While our method, based on the new formula
(\ref{gelmay1})  is original, certain of the ideas that
we have presented in this introduction have arisen in other forms before.
We now discuss the work of these other researchers.  \\

Weaker forms of Lemma \ref{manmol} have appeared before; here is the history, as best we can
determine it.  If $f$ has compact support away from $0$, in 1989, Seeger-Sogge \cite{SS} showed (\ref{xykt}) modulo
a remainder term that they must handle separately.  (We would not be able to handle this
remainder in the applications we seek.)  Next, in 1996, Tao  showed (\cite{Tao}. Proposition
3.1 (ii)) the case $j=k=0$ of Lemma \ref{manmol}, under the restriction that
$\hat{h}$ has compact support.  (Recall that $h(\xi) = f(\xi^2)$.  An assumption that
$\hat{h}$ has compact support would not be natural in our context.)    \\

Most relevantly, in 2006
Narcowich, Petrushev and Ward (\cite{narc1}, \cite{narc2}) showed a slight variant of Lemma \ref{manmol}
if ${\bf M} = S^n$, the sphere, provided $f$ had compact support away from $0$.  (In their variant,
$K_t(x,y)$ was the kernel not of $f(t^2\Delta)$ but rather of $f(t^2{\cal M})$,
where ${\cal M}$ is a particular first-order pseudodifferential operator which is similar to $\Delta$.
Specifically, $\Delta$ multiplies spherical harmonics of degree $l$ by $l(l+n-1)$, while ${\cal M}$
multiplies them by $l^2$.  This is a minor distinction, however.)

Narcowich, Petrushev and Ward do not discuss continuous wavelets, but use spectral theory
arguments to construct tight frames on $S^n$.  They then apply their variant of Lemma \ref{manmol}
for purposes similar to ours, including characterizations of Besov and Triebel-Lizorkin spaces
through frame coefficients on the sphere.  Their frames have been dubbed ``needlets'', and have been used by
statisticians and astrophysicists
to study cosmic microwave background radiation (CMB).  (See, for instance, \cite{mar}, \cite{baldi}, \cite{guil}
and the references therein.)  We present a detailed comparison of their approach to frames and ours in our
sequel article \cite{gmfr}.

Returning to our own approach, but still restricting to the sphere,
the most important cases to consider are the case in which $f$ has compact support away from $0$ (the
``needlet '' case, essentially considered by Narcowich, Petrushev and Ward), and the case in which $f(s) = s^r e^{-s}$ for some
integer $r \geq 1$.  In our sequel article \cite{gmfr}, we construct frames from the latter $f$;
we call these frames {\em Mexican needlets}.
Needlets and Mexican needlets each have their own advantages.  Needlets are a tight frame,
and frame elements at non-adjacent scales are orthogonal.  Mexican needlets, though not tight,
are nearly tight; they have the advantage that one can work with them directly on the sphere,
because of the formula (\ref{htapp0}).  (This formula is only for $r = 1$, but can be readily
generalized to general $r$.  As we said before, estimating the error in (\ref{htapp0}) is work in progress.)
Assuming this formula, Mexican needlets have strong Gaussian decay at each scale,
and do not oscillate (for small $r$), so they can be implemented directly on the sphere, which is
desirable if there is missing data (such as the ``sky cut'' of the CMB). \\

The statistical properties of needlets were investigated in \cite{mar}. Also,
the statistical properties of Mexican needlets are already being investigated, by the second author
in \cite{mayuncor}, and by Lan and Marinucci in \cite{lanmar}.

It would be worthwile to utilize both needlets and Mexican needlets in the analysis of CMB, and the
results should be compared.\\

A number of other researchers have studied wavelets and frames on manifolds.
In all of the works mentioned below, when orthonormal bases were constructed, they
are not known to give rise to a space-frequency analysis; and when frames were constructed,
they are not known to be tight or to give rise to a space-frequency analysis.\\

Let us begin by discussing earlier works on manifolds, which contain some ideas related to those
in this article.  In alphabetical order:\\
\begin{itemize}
\item  Antoine, Vandergheynst, and collaborators (\cite{ant}, \cite{bogant}) have constructed
smooth continuous
 wavelets on the sphere and related manifolds, by use of stereographic dilations (replacing the
 usual dilations), rotations (replacing the usual translations), and spherical convolution (replacing
the usual convolution).  They obtained frames by discretizing these continuous wavelets.
\item
   Coifman, Maggioni, and collaborators (\cite{coimag}, \cite{mag2}) used the heat equation on
manifolds for the rather different purpose of constructing orthonormal wavelet bases through a diffusion process,
 leading to a multiresolution analysis.  They exploit the idea (which they attribute to Stein) of $e^{-t\Delta}$
 being a sort of dilate of $e^{-\Delta}$.
\item
Freeden and collaborators (\cite{free1}, \cite{free2}) defined continuous wavelets
on the sphere $S^2$, and applied them to the geosciences.  Their continuous wavelets were of the form $f(t^2\Delta)$
for various $f$ (not all in the Schwartz space), although they did not formulate them in that manner.
One of their many examples was our Mexican hat wavelet, which they called the Gauss-Weierstrass wavelet
of order zero.   They did not have Lemma
\ref{manmol}, so they restricted to an (extensive) $L^2$ theory of continuous
wavelets.  In the context of $S^2$, they had, in particular, results equivalent to our (\ref{gelmay1}).
 \item
  Han  (\cite{han1}, \cite{han2}) constructed frames on general spaces of homogeneous type (including manifolds).
  His method is to discretize a discrete version of Calder\'on's formula
 in this general setting.  He also used the $T(1)$ theorem to estimate errors,
 \end{itemize}

The following researchers have also worked on wavelets and frames on manifolds.  They used methods which
seem unrelated to those in the present article.  In alphabetical order:
\begin{itemize}
 \item
  Dahlke (\cite{dahl}) constructed an analogue of Haar wavelets on Riemannian manifolds.
\item  Dahmen and Schneider (\cite{dahsch})  have used parametric lifings from standard bases on the unit cube
 to obtain wavelet bases on manifolds which are the disjoint union of smooth parametric images of the standard cube.
\item
 Schr\"oder and Sweldens (\cite{swel}) used a lifing scheme to build wavelets on manifolds.
 This lifting scheme uses no invariance properties, and regularity information is not easily obtained.
 \end{itemize}

\section{Applying the Spectral Theorem}\label{applying-the-spectral-theorem}
In this section, we give the proof of (\ref{gelmay1}), as well as the proof
of a discrete analogue, which will be used in our sequel article \cite{gmfr}
to construct nearly tight frames.  (The proofs are
quite elementary, and the reader who is willing to accept (\ref{gelmay1}) can go
on to the next section.)  Specifically, we shall show:

\begin{lemma}
\label{gelmaylem}
Let $T$ be a positive self-adjoint operator on a Hilbert space ${\cal H}$, and
let $P$ be the projection onto the null space of $T$.
Suppose $l \geq 1$ is an integer, $f_0 \in {\cal S}(\RR^+)$,
$f_0 \not\equiv 0$, and let $f(s) = s^l f_0(s)$.
Set $c = \int_0^{\infty} |f(t)|^2 \frac{dt}{t}$.
\begin{itemize}
\item[$(a)$] Then
for any $F \in {\cal H}$,
\begin{equation}
\label{strint}
\lim_{\varepsilon \rightarrow 0^+, N \rightarrow \infty}
\left[\int_{\varepsilon}^{N} |f|^2(tT) \frac{dt}{t}\right]F = c(I-P)F,
\end{equation}
Thus
\[ \int_0^{\infty} |f|^2(tT) \frac{dt}{t} :=
 \lim_{\varepsilon \rightarrow 0^+, N \rightarrow \infty}
\left[\int_{\varepsilon}^{N} |f|^2(tT) \frac{dt}{t}\right] \]
exists in the strong operator topology, and equals $c(I-P)$.
\item[$(b)$]
Suppose that $a > 0$  $(a \neq 1)$
is such that the {\em Daubechies condition} holds: for any $s > 0$,
\begin{equation}
\label{daubrep}
0 < A_a \leq \sum_{j=-\infty}^{\infty} |f(a^{2j} s)|^2 \leq B_a < \infty,
\end{equation}
Then
$\lim_{M, N \rightarrow \infty}
\left[\sum_{j=-M}^N |f|^2(a^{2j} T)\right]$
exists in the strong operator topology on ${\cal H}$; we denote this limit by
$\sum_{j=-\infty}^{\infty} |f|^2(a^{2j} T)$.  Moreover
\begin{equation}
\label{strsum1}
A_a(I-P) \leq \sum_{j=-\infty}^{\infty} |f|^2(a^{2j} T) \leq
B_a(I-P).
\end{equation}
\end{itemize}
\end{lemma}
{\bf Remark} $(b)$ is a discrete analogue of $(a)$, since the sum in (\ref{daubrep}) is a multiple
of a Riemann sum for the integral $\int_0^{\infty} |f(st)|^2 \frac{dt}{t} = c$, while
the spectral theorem will show that it is valid to replace $s$ in (\ref{daubrep}) by $T$, to obtain
(\ref{strsum1}).
\begin{proof}  We prove $(a)$ and $(b)$ together.
Let $T = \int_0^{\infty} \lambda dP_{\lambda}$
be the spectral decomposition of $T$; thus, in particular, $P = P_{\{0\}}$.

Observe that, by the spectral theorem, if $g$ is a
bounded Borel function on $\RR^+$, then $\|g(T)\| \leq
\sup_{s \geq 0} |g(s)|$.
It follows readily that the integrand in (\ref{strint}) is a family
of operators in ${\cal B}({\cal H})$ (indexed by $t$),
which depends continuously on $t \in [\varepsilon,N]$ (in the
norm topology on ${\cal B}({\cal H})$).
(Use the mean value theorem for $s$ in a suitable compact interval
and the rapid decay of $f$ at $\infty$.)
 Thus
$\int_{\varepsilon}^{N} |f|^2(tT) \frac{dt}{t}$
makes sense as a bounded operator on ${\cal H}$.  (As usual, it
is defined as the unique bounded operator $S$ with $ \langle  S\varphi,\psi  \rangle  =
\int_{\varepsilon}^{N}  \langle  |f|^2(tT)\varphi,\psi  \rangle  \frac{dt}{t}$
for all $\varphi,\psi \in {\cal H}$.)

For $0 < \varepsilon < N < \infty$, define $g_{\varepsilon,N}: [0,\infty) \rightarrow [0,\infty)$ by
\begin{equation}
\label{gepN}
g_{\varepsilon,N}(\lambda) = \int_{\varepsilon}^{N} |f|^2(t\lambda) \frac{dt}{t}.
\end{equation}
Then
\begin{equation}\notag
g_{\varepsilon,N}(T) =  \int_{\varepsilon}^{N} |f|^2(tT) \frac{dt}{t}.
\end{equation}
(This follows from an elementary application of Fubini's theorem to
$\int_{\varepsilon}^{N} \int_0^{\infty} |f|^2(t\lambda)
d \langle  P_{\lambda}\varphi,\psi  \rangle  \frac{dt}{t}$.)

For $M,N \geq 0$ we also define $h_{M,N}: [0,\infty) \rightarrow [0,\infty)$ by
\begin{equation}\notag
h_{M,N}(\lambda) = \sum_{j=-M}^N |f|^2(a^{2j} \lambda);
\end{equation}
and we also set
\begin{equation}
\label{hla}
h(\lambda) = \sum_{j=-\infty}^{\infty} |f|^2(a^{2j} \lambda).
\end{equation}

To prove the lemma it is enough to show that $g_{\varepsilon,N}(T) \rightarrow c(I-P)$
(strongly, as $\varepsilon \rightarrow 0^+$ and $N \rightarrow \infty$) and
that $h_{M,N}(T) \rightarrow h(T)$
(strongly, as $M,N \rightarrow \infty$).
Indeed, $(a)$ would then be immediate. $(b)$ would then also follow at once
from (\ref{daubrep}), which implies, by the spectral theorem, that
$A_a(I-P) \leq h(T) \leq B_a(I-P)$.\\

To establish these conclusions, we first note that this strong convergence
need only be proved on $(I-P){\cal H}$, since all the operators vanish identically
on $P{\cal H}$.  Next we note that
for any $\varepsilon, M, N$,
 $\|g_{\varepsilon,N}\|_{\infty}
\leq c$, $\|h_{M,N}\|_{\infty} \leq B_a$, and $\|h\|_{\infty} \leq B_a$,
whence
$||g_{\epsilon,N}(T)||
\leq c$, $||h_{M,N}(T)||_{\infty} \leq B_a$, and $||h(T)|| \leq B_a$,

Thus the needed strong convergence need only be proved on a dense subset
of $(I-P){\cal H}$.

Set
\[ V = \bigcup_{0 < \eta < L < \infty} P_{[\eta,L]}\mathcal{H};\]
observe that $V$ is dense in $(I-P)\mathcal{H}$ (since $P = P_{\{0\}}$).

Thus, it suffices to show the following: fix $0 < \eta < L < \infty$,
and say $F \in P_{[\eta,L]}\mathcal{H}$.  Then
$g_{\varepsilon,N}(T)F \rightarrow cF$ and
$h_{M,N}(T)F \rightarrow h(T)F$.

This, however, is immediate from the spectral theorem
and the evident facts that $g_{\epsilon,N}
\rightarrow c$ and $h_{M,N} \rightarrow h$, {\em uniformly}
on $[\eta,L]$.   Although this uniform convergence is easily shown,
for later purposes, we carefully
express it quantitatively in the next lemma.  (Note that we may assume
$a > 1$; otherwise replace it by $1/a$.)\end{proof}

\begin{lemma}
\label{uncon}
Notation as in Lemma $\ref{gelmaylem}$, and as in equations
$(\ref{gepN})$  through $(\ref{hla})$.   Suppose $J \geq 1$ is an integer,
and let $M_J = \max_{r > 0}|r^J f(r)|$.  Suppose $0 < \eta < L < \infty$,
and let $I$ be the closed interval $[\eta,L]$.  Let $\|\:\|$ denote
the sup norm on $I$.
\begin{itemize}
\item[$(a)$] If $0 < \varepsilon < N$, then
\[ \|g_{\varepsilon,N} - c\| \leq c_L \varepsilon^{2l} + \frac{C_{\eta}}{N^{2J}}, \]
where we may take $c_L = L^{2l}\|f_0\|^2_{\infty}/(2l)$,
and $C_{\eta} = M_J^2/(2J\eta^{2J})$.
\item[$(b)$] If $M,N \geq 0$, and $a > 1$, then
\[ \|h_{M,N} - h\| \leq \frac{c^{\prime}_L}{a^{4Ml}} +
\frac{C^{\prime}_{\eta}}{a^{4NJ}}. \]
where we may take
$c^{\prime}_L = (L^{2l}\|f_0\|^2_{\infty})/(a^{4l}-1)$, and
$C^{\prime}_{\eta} = M_J^2/[(a^{4J}-1)\eta^{2J}]$.
\end{itemize}
\end{lemma}
\begin{proof}
Say $s \in I$.  Then
\[ |g_{\varepsilon,N}(s) - c| = \int_0^{\varepsilon}|f|^2(st)\frac{dt}{t}
+ \int_N^{\infty}|f|^2(st)\frac{dt}{t}. \]
(a) follows from noting $|f|^2(st) \leq (Lt)^{2l}\|f_0\|^2_{\infty}$ in the
first integral, and that $|f|^2(st) \leq M_J^2/(\eta t)^{2J}$ in the second.
Similarly,
\[ |h_{M,N}(s) - h(s)| = \sum_{j<-M} |f|^2(a^{2j}s) + \sum_{j>N} |f|^2(a^{2j}s). \]
(b) follows from noting $|f|^2(a^{2j}s) \leq (La^{2j})^{2l}\|f_0\|^2_{\infty}$
in the first summation, and that $|f|^2(a^{2j}s) \leq M_J^2/(\eta a^{2j})^{2J}$
in the second.  This completes the proof.
\end{proof}

We can now express the strong convergence in Lemma \ref{gelmaylem} in the
following very quantitative manner:

\begin{lemma}
\label{gelmayquant}
Notation as in Lemmas \ref{gelmaylem} and \ref{uncon}, and again let
$T = \int_0^{\infty} \lambda dP_{\lambda}$ be the spectral decomposition of $T$.
Then for any $F \in {\mathcal H}$, we have:
\begin{equation}
\label{quant1}
\left\|\left[g_{\varepsilon,N}(T) - c\right]F\right\| \leq
\left[c_L \varepsilon^{2l} + \frac{C_{\eta}}{N^{2J}}\right]\|F\| +
2c \left\|(I - P_{[\eta,L]})F\right\|,
\end{equation}
and, if $a > 1$,
\begin{equation}
\label{quant2}
\left\|\left[h_{M,N}(T) - h(T)\right]F\right\| \leq \left[\frac{c^{\prime}_L}{a^{4Ml}} +
\frac{C^{\prime}_{\eta}}{a^{4NJ}}\right]\|F\| +
2B_a \left\|(I - P_{[\eta,L]})F\right\|.
\end{equation}
\end{lemma}
\begin{proof} For (\ref{quant1}), we need only substitute $F = P_{[\eta,L]}F + (I-P_{[\eta,L]})F$
in the left side.  This gives
\begin{eqnarray*}
\|\left[g_{\varepsilon,N}(T) - c\right]F\| & \leq & \|\left[g_{\varepsilon,N}(T) - c\right]P_{[\eta,L]}F\|
+ \|g_{\varepsilon,N}(T)\|\:\|(I - P_{[\eta,L]}) F\|
+ c\|(I - P_{[\eta,L]}) F\| \\
& \leq  & \|[g_{\varepsilon,N}(T) - c]\chi_{[\eta,L]}\|_{\infty}\|F\|
+ 2c\|(I - P_{[\eta,L]}) F\|,
\end{eqnarray*}
which, because of Lemma \ref{uncon}, establishes (\ref{quant1}).  The proof of (\ref{quant2})
is entirely analogous.
 \end{proof}

(\ref{quant1}) and (\ref{quant2}) are of significance for numerical calculations.  Say,
for instance, that $F = (I-P)F$, and
one wants to compute $h(T)F$ to a certain precision.  This involves summing
an infinite series, so one instead seeks to compute $[h_{M,N}(T)]F$ for $M, N$ large enough;
how large must one take them to be?  One first chooses $\eta,L$ so that the
second term on the right side of (\ref{quant2}) is
very small.  Then one chooses $M, N$ to make the first term on the right side of
(\ref{quant2}) very small as well.

We will return to this point in our discussion of space-frequency analysis, in our
sequel article \cite{gmfr}.

\section{Preliminaries on Manifolds}\label{preliminaries-on-manifolds}

For the rest of the article, $({\bf M},g)$ will denote a smooth, compact,
connected, oriented Riemannian manifold of dimension $n$,
and $\mu$ will denote the measure on ${\bf M}$ arising from integration with respect to the
volume form on ${\bf M}$.  In this section we assemble
several well-known facts about analysis on ${\bf M}$ (preceded, below, by bullets),
which we shall need in this article and in sequel articles.
\ \\
For $x,y \in {\bf M}$, we let $d(x,y)$ denote the infimum of the lengths of
all piecewise $C^1$ curves joining $x$ to $y$; then $({\bf M},d)$ is evidently
a metric space.  It is well-known (see, e.g., \cite{Milnor}) that there is a geodesic
joining $x$ to $y$ with length $d(x,y)$, but this fact, though basic, is not
so relevant for this article.  Most of what we need to know about the metric
$d$ is contained in the simple proposition which follows.

For $x \in {\bf M}$, we let $B(x,r)$ denote the ball $\{y: d(x,y) < r\}$.
\begin{proposition}
\label{ujvj}
Cover ${\bf M}$ with a finite collection of open sets $U_i$  $(1 \leq i \leq I)$,
such that the following properties hold for each $i$:
\begin{itemize}
\item[$(i)$] there exists a chart $(V_i,\phi_i)$ with $\overline{U}_i
\subseteq V_i$; and
\item[$(ii)$] $\phi_i(U_i)$ is a ball in $\RR^n$.
\end{itemize}
Choose $\delta > 0$ so that $3\delta$ is a Lebesgue number for the covering
$\{U_i\}$.  Then, there exist $c_1, c_2 > 0$ as follows:\\
For any $x \in {\bf M}$, choose any $U_i \supseteq B(x,3\delta)$.  Then, in
the coordinate system on $U_i$ obtained from $\phi_i$,
\begin{equation}\notag
d(y,z) \leq c_2|y-z|
\end{equation}
for all $y,z \in U_i$; and
\begin{equation}\notag
c_1|y-z| \leq d(y,z)
\end{equation}
for all $y,z \in B(x,\delta)$.
\end{proposition}
\begin{proof} Say $y,z \in U_i$.
We work in
the coordinate system on $U_i$ obtained from $\phi_i$.  Then
$d(y,z)$ is at most the length (in the Riemannian metric)
of the straight line joining
$y$ to $z$, which is $\leq c_2|y-z|$ for some $c_2$.  (By assumption (ii),
that straight line is contained in $U_i$.)
On the other hand, if $y,z \in B(x,\delta) \subseteq U_i$,
we may take a sequence of piecewise $C^1$ curves $\gamma_k$, joining $y$ to $z$, whose lengths
$l(\gamma_k)$ approach $d(y,z)$ as $k \rightarrow \infty$.  Surely $d(y,z) < 2\delta$.
Thus, for some $N > 0$, if $k > N$, then
$l(\gamma_k) < 2\delta$.  Therefore each point on $\gamma_k$ is at distance
at most $2\delta$ from $y$, hence at most $3\delta$ from $x$.
  Thus $\gamma_k \subseteq U_i$.  Letting $\|\gamma_k^{\prime}(t)\|$
denote the length of the tangent
vector $\gamma_k^{\prime}(t)$ in the Riemannian metric, we see that
\[ l(\gamma_k) =
\int_0^1 \|\gamma_k^{\prime}(t)\|dt \geq
c_1\int_0^1 |\gamma_k^{\prime}(t)| dt \geq c_1|z-y|,\]
since $z-y = \int_0^1 \gamma_k^{\prime}(t) dt$.  Letting $k \rightarrow \infty$, we see that
$d(y,z) \geq c_1|y-z|$ as well.
This completes the proof.
\end{proof}

We fix collections $\{U_i\}$, $\{V_i\}$, $\{\phi_i\}$ and also $\delta$ as in
Proposition \ref{ujvj}, once and for all.
\begin{itemize}
\item {\em Notation as in Proposition \ref{ujvj},
there exist $c_3, c_4 > 0$, such that
\begin{equation}
\label{ballsn}
c_3r^n \leq \mu(B(x,r)) \leq c_4r^n
\end{equation}
whenever $x \in {\bf M}$ and $0 < r \leq \delta$, and such that
\begin{equation}
\label{ballsn1}
c_3 \delta^n \leq \mu(B(x,r)) \leq \mu({\bf M}) \leq c_4r^n
\end{equation}
whenever $x \in {\bf M}$ and $r > \delta$.}\\
\ \\
To see (\ref{ballsn}), note that, since the collection
$\{U_i\}$ is finite, we may fix $i$ and prove it for all $x$ with
$B(x,3\delta) \subseteq U_i$.

We work
in the coordinate system on $U_i$ obtained from $\phi_i$;
in that coordinate system, $U_i$ is a Euclidean ball, say $\{y: |y-x_0| < R\}$.
(See Proposition \ref{ujvj}).  By compactness and a simple contradiction
argument, there is an $\eta > 0$ such that, for all $x$ with
$B(x,3\delta) \subseteq U_i$, one has that $|x-x_0| < R - \eta$.
Accordingly, for such an $x$, if $|y-x| < \eta$, then $y \in U_i$.

Thus, by Proposition \ref{ujvj}, we have that
\[ \{y: |y-x|< \min(r/c_2,\eta)\} \subseteq B(x,r) \subseteq \{y: |y-x|< r/c_1\}, \]
for all $r < \delta$.
(\ref{ballsn}) now follows from the fact that
the determinant of the metric tensor $g$ is bounded above and below on $U_i$.
For (\ref{ballsn1}), one need only note that if $r > \delta$, then
$\mu(B(x,r)) \geq \mu(B(x,\delta)) \geq c_3\delta^n$, while
$\mu(B(x,r)) \leq \mu({\bf M}) \leq [\mu({\bf M})/\delta^n]\delta^n \leq
[\mu({\bf M})/\delta^n]r^n$.
\item {\em
  $({\bf M},d,\mu)$ is a space of homogeneous type, in the sense of
\cite{CoifWei71}.}\\
\ \\
Indeed, $d$ is a metric and $\mu$ is a positive Borel
measure, so one only needs to check the doubling condition: $\mu(B(x,2r)) \leq
C\mu(B(x,r))$ with $C$ independent of $x,r$.  But this is immediate from
(\ref{ballsn}) and (\ref{ballsn1}).
\item {\em
  For any $N > n$ there exists $C_N$ such that
\begin{equation}
\label{ptestm}
\int_{\bf M} [1 + d(x,y)/t]^{-N} d\mu(y) \leq C_N t^n
\end{equation}
for all $x \in {\bf M}$ and $t > 0$.}\\
\ \\
(\ref{ptestm}) is proved by the ``dyadic annulus'' method.  Fix $x, t$
and let $A_j = B(x,2^jt) \setminus B(x,2^{j-1}t)$, so that, by (\ref{ballsn}) and (\ref{ballsn1}),
$\mu(A_j) \leq c_42^{nj}t^n$. (\ref{ptestm}) now follows at once,
if one breaks up the integral in (\ref{ptestm}) into integrals over $B(x,t), A_1, A_2, \ldots$,
and notes that $\sum_{j=0}^{\infty} 2^{(n-N)j} < \infty$.
\item {\em
  For any $N > n$ there exists $C_N^{\prime}$ such that
\begin{equation}
\label{ptestm1}
\int_{d(x,y) \geq t} d(x,y)^{-N} d\mu(y) \leq C_N^{\prime} t^{n-N}
\end{equation}
for all $x \in {\bf M}$ and $t > 0$.}\\
\ \\
(\ref{ptestm1}) follows at once from (\ref{ptestm}), once we observe that,
if $d(x,y) \geq t$, then\\
$[d(x,y)/t]^{-N} \leq C[1+d(x,y)/t]^{-N}$, if $C = 2^N$.
\item{\em
  For any $N > n$ there exists $C_N^{\prime \prime}$ such that
\begin{equation}
\label{ptestm2}
\int_{\bf M} [1 + d(x,z)/t]^{-N} [1 + d(z,y)/t]^{-N}d\mu(z) \leq
C_N^{\prime \prime} t^n[1 + d(x,y)/t]^{-N}
\end{equation}
for all $x,y \in {\bf M}$ and $t > 0$.}\\
\ \\
To see this, break up the integral into integrals over $H_1 = \{z: d(x,z) \leq d(y,z)\}$
and $H_2 = \{z: d(x,z) > d(y,z)\}$ (which, by the way, are hemispheres if ${\bf M}$ is a
round sphere and $x \neq y$).  By symmetry we need only estimate the integral
over $H_1$. But if $z$ is
in $H_1$, $d(x,y) \leq 2d(z,y)$, so $[1 + d(z,y)/t]^{-N}
 \leq C[1 + d(x,y)/t]^{-N}$ (where $C = 2^N$).  Thus the integral over $H_1$ is no greater
than
$C[1 + d(x,y)/t]^{-N}\int_{H_1} [1 + d(x,z)/t]^{-N} d\mu(z))$.  Estimating
the latter integral through (\ref{ptestm}), we obtain (\ref{ptestm2}).
\item{\em
  For all $M, t > 0$, and for all
$E \subseteq {\bf M}$ with diameter less than $Mt$, if
$x_0 \in E$, then one has that
\begin{equation}
\label{alcmpN}
\frac{1}{M+1}[1+d(x,y)/t] \leq [1+d(x_0,y)/t] \leq (M+1)[1+d(x,y)/t]
\end{equation}
for all $x \in E$ and all $y \in {\bf M}$.}\\
\ \\
This is true simply because $d$ is a metric.
\end{itemize}
\section{Kernels}\label{kernels}
$\Delta$ will now denote the Laplace-Beltrami operator on ${\bf M}$ (equal to
$-d^*d$, where $\ ^*$ is taken with respect to the given Riemannian metric).
We apply Lemma \ref{gelmaylem} to $T=\Delta$.

In order to carry out the plan explained in the introduction to this article,
we must study the kernel $K_{\sqrt t}(x,y)$ of the operator $f(t\Delta)$ for
$f \in {\mathcal S}(\RR^+)$, and we do so in this section.  Before proving
the crucial Lemma \ref{manmol}, we will present some well-known information about
$K_{\sqrt t}(x,y)$ for large $t$ and also, off the diagonal, for small $t$.  (This information will be
preceded, below, by $\rhd$ signs.)\\
Concerning the Laplace-Beltrami operator $\Delta$, we first recall:
\begin{itemize}
\item   $\Delta$, as an operator on $C^{\infty}({\bf M})$,
has an orthonormal basis of eigenfunctions $\{u_l: 0 \leq l < \infty\}$;
all are in $C^{\infty}({\bf M})$.  We may, and do, choose the $u_l$ to
be real-valued.  We order the $u_l$ so that the corresponding
eigenvalues $\lambda_l$ form a non-decreasing sequence.  Then $u_0$ is constant,
$\lambda_0 = 0$ and all other $\lambda_l > 0$.
\vspace{.3cm}

This easily implies:
\item
 $\Delta$, as an operator on $C^{\infty}({\bf M})$, has a unique
extension to a self-adjoint operator on $L^2({\bf M})$ (which we also
denote by $\Delta$).  Its domain is $\{F= \sum a_l u_l \in L^2({\bf M}):
\sum |\lambda_l a_l|^2 < \infty\}$, and for such an $F$, $\Delta F =
\sum \lambda_l a_l u_l$.
\vspace{.3cm}

 Of great importance is {\em Weyl's Lemma}, which says \cite{Hor68}, in sharp form:
\item For $\lambda > 0$, let $N(\lambda)$ denote the number of eigenvalues of
$\Delta$ which are less than or equal to $\lambda$ (counted with respect to
multiplicity).  Then for some $c > 0$,
$N(\lambda) = c\lambda^{n/2} + O(\lambda^{(n-1)/2})$.
\vspace{.3cm}

Since $N(\lambda_l) = l+1$, we conclude:
\item
  For some constants $c_1, c_2 > 0$, we have
$c_1 l^{2/n} \leq \lambda_l \leq c_2 l^{2/n}$ for all $l$.
\vspace{.3cm}

Since $\Delta^m u_l = \lambda_l^m u_l$, and $\Delta^m$ is an elliptic differential
operator of degree $2m$, Sobolev's lemma, combined with the last fact, implies:
\item For any integer $k \geq 0$, there exists $C_k, \nu_k > 0$ such that
$\|u_l\|_{C^k({\bf M})} \leq C_k (l+1)^{\nu_k}$.
 \vspace{.2cm}

(In fact, by Sobolev's lemma, we may, for any $\varepsilon > 0$, take
$\nu_k = (2k+n+\varepsilon)/2n$.  By \cite{Sogge93}, Lemma 4.2.4, with
$\lambda = \lambda_l$ in that lemma, we may in fact take $\nu_0 = (n-1)/n$.)
\vspace{.3cm}

 From these  facts one sees at once:
\item The mapping $\sum a_l u_l \rightarrow (a_l)_{l \geq 0}$ gives a Fr\'echet space
isomophism of $C^{\infty}({\bf M})$ with the space of rapidly decaying
sequences.
\end{itemize}
For the rest of this section, say $f \in {\mathcal S}(\RR^+)$.  We conclude:
\begin{itemize}
\item[
$\rhd$] For $t > 0$, $x,y \in {\bf M}$, let (for the rest of this section)
\begin{equation}
\label{kerexp}
K_{{\sqrt t}}(x,y) = \sum_{l=0}^{\infty} f(t\lambda_l)u_l(x)u_l(y).
\end{equation}
Then $K_{\sqrt t}$ is the kernel of the operator $f(t\Delta)$, in the
sense that if $F \in L^2({\bf M})$, then
\begin{equation}\notag
[f(t\Delta) F](x) = \int_{\bf M} K_{\sqrt t}(x,y) F(y) d\mu(y),
\end{equation}
and $K_{\sqrt t}(x,y)$ is smooth in $(t,x,y)$ (for $t > 0$, $x, y \in {\bf M}$).\\
\ \\
From (\ref{kerexp}), our estimates on the $\lambda_l$ and the $\|u_l\|_{C^k}$,
and the rapid decay of Schwartz functions, we conclude:
\item[
$\rhd$] If $f(0) = 0$, then for any $M, N \geq 0$,
\[ \lim_{t \rightarrow \infty} t^M \frac{\partial^N}{\partial t^N} K_{\sqrt t} = 0 \]
in $C^{\infty}({\bf M} \times {\bf M})$.
\vspace{.3cm}

Note also, that whenever $l > 0$, $\int_{\bf M} u_l d\mu
= C \int_{\bf M} u_l \overline{u_0} d\mu  = 0$.  Accordingly:\\
\item[
$\rhd$]  If $f(0) = 0$, then $\int_{\bf M} K_{\sqrt t} (x,y) d\mu(x) = 0$
for all $y \in {\bf M}$, and $\int_{\bf M} K_{\sqrt t} (x,y) d\mu(y) = 0$
for all $x \in {\bf M}$.
\end{itemize}

Next we discuss the behavior of $K_{\sqrt t}$, as $t \rightarrow 0^+$.
For this, we utilize some of the theory of pseudodifferential
operators.

We will need some facts about symbols on $\RR$.
For $m \in \RR$, let $S^m_1(\RR)$ denote the space of standard
symbols $p(\xi)$ of order $m$, which depend only on the ``dual variable" $\xi$.
Let $\{\|\:\|_{m,N}\}$ denote the
natural nondecreasing family of seminorms defining the Fr\'echet space topology
of $S^{m}_{1}(\RR)$; thus
\[ \|p\|_{m,N} = \sum_{0 \leq j \leq N}\sup_{\xi} \left[(1+|\xi|)^{j-m}|p^{(j)}(\xi)|\right] \]
for $p \in S^{m}_{1}(\RR)$.
If $G \in S^m_1(\RR)$, let $G_t(\xi) = G(t\xi)$.
It is evident that $G_t \in S^{m}_{1}(\RR)$.
In fact we have:
\begin{itemize}
\item {\em Say that $G \in S^m_1(\RR)$.  Let $k$ be the least integer
which is greater than or equal to $m$; if $k > 0$,
suppose further that $G$ vanishes to order at least $k$ at $0$.
Then for any $N$, there exists $C > 0$ such that
\begin{equation}
\label{tmdep}
\|G_t\|_{m,N} \leq Ct^m
\end{equation}
whenever $0 < t < 1$}.
\end{itemize}
To see this, say $j \geq 0$ is an integer.
One needs only note:
\begin{itemize}
\item[$(i)$] if $0 \leq j < m$ (so that $k > 0$) and $t|\xi| \leq 1$, then
\[ |(G_t)^{(j)}(\xi)| \leq Ct^j (t^{k-j}|\xi|^{k-j})
\leq Ct^j (t^{m-j}|\xi|^{m-j}) \leq C t^m (1 + |\xi|)^{m-j};\]
\item[$(ii)$] if $0 \leq j < m$ and $t|\xi| > 1$, then
\[ |(G_t)^{(j)}(\xi)| \leq Ct^j (1 + t|\xi|)^{m-j}
\leq Ct^j (t^{m-j}|\xi|^{m-j}) \leq C t^m (1 + |\xi|)^{m-j};\]
\item[$(iii)$] if $j \geq m$, then
\[|(G_t)^{(j)}(\xi)| \leq Ct^j (1+|t\xi|)^{-(j-m)} \newline \leq Ct^m(1+|\xi|)^{-(j-m)}.\]
\end{itemize}
This proves the claim.\\
\ \\
In particular, say $G \in S^1_1(\RR)$ is arbitrary.
Then $G-G(0) \in S^{1}_{1}(\RR)$, and vanishes to order at least $1$ at $0$.
Accordingly, for any $N$,
$\|G_t-G(0)\|_{1,N} \leq Ct$ for $0 < t < 1$, so that
$G_t \rightarrow G(0)$ in $S^1_{1}(\RR)$ as $t \rightarrow 0^+$.\\
 \\
We return now to ${\bf M}$, where we need to look at the class $OPS^m_{1,0}({\bf M})$
of pseudodifferential operators of order $m \in [-\infty,\infty)$.
As is familiar, $T: C^{\infty}({\bf M}) \to C^{\infty}({\bf M})$ is in $OPS^m_{1,0}({\bf M})$ provided
that the following conditions hold for $\varphi, \psi \in C^{\infty}({\bf M})$:
\begin{enumerate}
\item  If $\supp \varphi \cap \supp \psi = \oslash$, then the operator $\varphi T \psi$ has a smooth kernel; and
\item If $\supp \varphi \cup \supp \psi$ is contained in a chart $(V,\Phi)$, then $\varphi T \psi$
is the pullback to ${\bf M}$ of a pseudodifferential operator $\Phi_*(\varphi T \psi)
\in OPS^m_{1,0}(\RR^n)$.\\
\end{enumerate}
(Of course, here, $(\varphi T \psi)F = \varphi T(\psi F)$.)  One places a Fr\'echet space
structure on $OPS^m_{1,0}({\bf M})$ in a natural manner.  (A brief sketch: First note that
$OPS^{-\infty}({\bf M})$ is the space of operators with smooth kernels, so it has
a natural Fr\'echet space structure, inherited from $C^{\infty}({\bf M} \times {\bf M})$.
For other $m$, one chooses a finite atlas $\{W_k\}$ on ${\bf M}$ with the property
that if two charts in the atlas intersect, their union is contained in a chart.  One chooses
a partition of unity $\{\varphi_k\}$ subordinate to this atlas.  One notes that
if $T \in OPS^m_{1,0}(\RR^n)$, then $T = \sum_{i,j} \varphi_i T \varphi_j$.  One notes that
if $W_i \cap W_j = \oslash$, then $\varphi_i T \varphi_j \in OPS^{-\infty}({\bf M})$,
a Fr\'echet space; while if $W_i \cap W_j \not= \oslash$, then $W_i \cap W_j
\subseteq V$ for some chart $(V,\Phi)$, and $\Phi_*(\varphi T \psi) \in OPS^m_{1,0}(\RR^n)$, also
a Fr\'echet space.  Finally one defines seminorms on
$OPS^m_{1,0}(\RR^n)$ of the form $\sum_{i,j}\|\varphi_i T \varphi_j\|$, where, in the summation,
one uses appropriate seminorms coming from $OPS^{-\infty}({\bf M})$ if $W_i \cap W_j = \oslash$,
or from $OPS^m_{1,0}(\RR^n)$ if $W_i \cap W_j \not= \oslash$.  The Fr\'echet space topology
thereby placed on $OPS^m_{1,0}({\bf M})$ is independent of all choices made.)\\

One has the following theorem of Strichartz (\cite{Stric72}, or Theorem 1.3, page 296, of \cite{Tay81}):
\begin{itemize}
\item {\em If $p(\xi) \in S^{m}_{1}(\RR)$, then $p(\sqrt{\Delta}) \in OPS^m_{1,0}({\bf M})$.}
\end{itemize}
In fact, the map $p \rightarrow p(\sqrt{\Delta})$ is continuous from $S^{m}_{1}(\RR)$
to $OPS^m_{1,0}({\bf M})$.  Indeed, by the closed graph theorem for Fr\'echet
spaces, it is enough to observe that if $u \in C^{\infty}({\bf M})$, then
the maps $p \rightarrow  \langle  p(\sqrt{\Delta})u,u_l  \rangle $ are continuous from
$S^{m}_{1}(\RR)$ to $\CC$ for every $l$, and this is clear.\\
\ \\
As a consequence, if $\varphi, \psi \in C^{\infty}({\bf M})$ are as in \#1 above,
then the map from $S^{m}_{1}(\RR)$ to $OPS^{-\infty}({\bf M})$,
which takes $p$ to $\varphi p(\sqrt{\Delta}) \psi$, is continuous.
If $\varphi, \psi \in C^{\infty}({\bf M})$ are as in \#2 above, then the map from
$S^{m}_{1}(\RR)$ to $OPS^m_{1,0}(\RR^n)$, which takes $p$ to $\Phi_*(\varphi p(\sqrt{\Delta}) \psi)$
is continuous.\\
\ \\
As usual, $f \in {\mathcal S}(\RR^+)$; let
$G(\xi) = f(\xi^2)$.   Then $G \in {\mathcal S}(\RR)$.  (In fact, if we allow $f$
to vary, the map $f \rightarrow G$ is evidently a bijection between ${\mathcal S}(\RR^+)$
and the space of even Schwartz functions on $\RR$.)  Now
$G_{\sqrt t}(\xi) = f(t \xi^2)$,
and $G_{\sqrt t}(\sqrt{\Delta}) = f(t \Delta)$.
Since $G_{\sqrt t} \rightarrow G(0)$ in $S^1_{1}(\RR)$ as $t \rightarrow 0^+$,
we infer:
\begin{itemize}
\item {\em $f(t\Delta) \rightarrow f(0)I$ in
$OPS^1_{1,0}({\bf M})$ as $t \rightarrow 0^+$. }
\end{itemize}
Let $D$ denote the diagonal of ${\bf M} \times {\bf M}$.  We can now show:

\begin{itemize}
\item[$\rhd$] {\em For any $N > 0$,
\[ \lim_{t \rightarrow 0}  \frac{\partial^N}{\partial t^N} K_{\sqrt t} = 0 \]
in $C^{\infty}(({\bf M} \times {\bf M}) \setminus D)$.}
\end{itemize}
To prove this, we adapt the arguments of \cite{Tay81}, page 313.
Say $\varphi, \psi \in C^{\infty}({\bf M})$, have disjoint supports.
Suppose further that $\varphi \equiv 1$ in an open set $U$ and $\psi \equiv 1$ in an open
set $V$.

It is enough to show that, for any $C^{\infty}$ differential operator $Y$ on ${\cal M}$,
acting in the $y$ variable,
\[ \lim_{t \rightarrow 0}  Y \frac{\partial^N}{\partial t^N} K_{\sqrt t}(x,y)
\mbox{ (regarded as a function of } x) = 0 \]
in $C^{\infty}(U)$, uniformly for $y \in V$.

But if $x \in U$ and $y \in V$, then
\[ Y \frac{\partial^N}{\partial t^N} K_{\sqrt t}(x,y)
= \varphi(x)\left[Yf^{(N)}(t\Delta)(\psi \Delta^N \delta_y)\right](x) = \left[S_t(w_y)\right](x), \mbox{ say }, \]
where $S_t$ is the pseudodifferential operator $\varphi Yf^{(N)}(t\Delta)\psi$,
and $w_y = \Delta^N \delta_y$.

For some $s > 0$, the set $\{w_y: y \in V\}$ is a bounded subset of
$H^{-s}({\bf M})$.
Also, as $t \rightarrow 0^+$, $f^{(N)}(t\Delta) \rightarrow f^{(N)}(0)I$ in
$OPS^1_{1,0}({\bf M})$, so $Y f^{(N)}(t\Delta) \rightarrow f^{(N)}(0)Y$ in
$OPS^{k+1}_{1,0}({\bf M})$, if $k = \deg Y$.
But $\varphi, \psi$ have disjoint supports, so
the map $R \rightarrow \varphi R \psi$ is continuous from $OPS^{k+1}_{1,0}({\bf M})$
to $OPS^{-\infty}_{1,0}({\bf M})$.  Therefore
$S_t \rightarrow \varphi \left[f^{(N)}(0) Y\right] \psi \equiv 0$
in $OPS^{-\infty}_{1,0}({\bf M})$.  Thus:

\[ S_t w_y \rightarrow 0 \mbox{ in }
C^{\infty}({\bf M}), \mbox{ uniformly for }y \in V,  \]

as desired.\\
\ \\
Applying the mean value theorem in the $t$ variable repeatedly to the last fact about
$K_{\sqrt t}$, we see:
\begin{itemize}
\item[
$\rhd$]  Let $E$ be any fixed
compact subset of $({\bf M} \times {\bf M}) \setminus D$, and let ${\mathcal U}$
be the interior of $E$.  Then for any $k, N$ there exists $C_{k,N}$ such that
\[ \|K_{\sqrt t}\|_{C^k({\mathcal U})} \leq C_{k,N} t^N \]
whenever $0 < t < 1$.
\end{itemize}

So far, nearly everything in this section has been well-known, but now we must consider
the behavior of $K_t$ near the diagonal for small $t$.  As we have explained and motivated in the
introduction, this behavior is described by (\ref{xykt}):
\begin{lemma}
\label{manmol}
Say $f(0) = 0$.  Then
for every pair of
$C^{\infty}$ differential operators $X$ $($in $x)$ and $Y$  $($in $y)$ on ${\bf M}$,
and for every integer $N \geq 0$, there exists $C_{N,X,Y}$ as follows.  Suppose
$\deg X = j$ and $\deg Y = k$.  Then
\begin{equation}
\label{diagest}
t^{n+j+k} \left|\left(\frac{d(x,y)}{t}\right)^N XYK_t(x,y)\right| \leq C_{N,X,Y}
\end{equation}
for all $t > 0$ and all $x,y \in {\bf M}$.
\end{lemma}
\begin{proof}  Of course $d$ is bounded on ${\bf M} \times {\bf M}$.
Thus, by what we already know about $K_t$, it
suffices to prove (\ref{diagest}) for $0 < t < 1$.  In fact, with
notation as in Proposition \ref{ujvj},
it suffices to show that
(\ref{diagest}) holds whenever $0 < t < 1$ and $d(x,y) < \delta/3$.
Cover ${\bf M}$ by a finite collection of balls $\{B(z_l,\delta/3): 1 \leq l \leq J\}$.
Then any pair of points $(x,y)$ with $d(x,y) < \delta/3$ lie together in one of the
balls $B(z_l,2\delta/3)$.  Thus, it suffices to show that
(\ref{diagest}) holds whenever $0 < t < 1$ and $x,y \in B(z_l,2\delta/3)$ for some
$l$.  Moreover, we may fix a positive integer $M$ and prove that
(\ref{diagest}) holds for all $N \leq M$. \\
\ \\
Now $K_t$ is the kernel of $f(t^2\Delta) = G(t\sqrt{\Delta})$, where
$G(\xi) = f(\xi^2)$.
We claim that it is enough to prove (\ref{diagest}) in each of the
following two cases:
\begin{itemize}
\item[$(i)$] supp$\widehat{G} \subseteq (-1,1)$; and
\item[$(ii)$]$G$ vanishes to order at least $M$ at $0$.
\end{itemize}
Indeed, any even $G$ in ${\mathcal S}(\RR)$ can be written as the sum of
two even functions $G_1$ and $G_2$, where
$G_1$ is of type (i) and $G_2$ is of type (ii).  (To see this, say that,
for $0 \leq l \leq M-1$, $G^{(l)}(0) = a_l$.  It is enough to show that there exists
an even function $G_1$ with supp$\widehat{G_1} \subseteq (-1,1)$, such that
$G_1^{(l)}(0)= a_l$,
for $0 \leq l \leq M-1$, for then we can set $G_2=G-G_1$.  For this, see
Lemma \ref{ccamplem} in the Appendix (Section \ref{cclem}).)\\

In case (i), note that, by Huygens' principle, the support of $K_t$, the kernel of
\begin{equation}\notag
G(t{\sqrt \Delta}) = c\int_{-\infty}^{\infty} \hat{G}(s) e^{-ist{\sqrt \Delta}} ds,
\end{equation}
is contained in $\{(x,y): d(x,y) \leq t\}$.  Thus, in this case, we may take $M=0$.

In either case (i) or case (ii),
it is sufficient to show that, for every $\varphi, \psi \in
 C_c^{\infty}(B(z_l,\delta))$, we have that

 \begin{align}\notag
 t^{n+j+k} \left(\frac{d(x,y)}{t}\right)^N  \left|XY\left[\varphi(x)K_t(x,y)\psi(y)\right] \right| \leq C
 \end{align}

whenever $\deg X = j$ and $\deg Y = k$, for all $0 < t < 1$, all $N \leq M$, and all
$x,y \in B(z_l,\delta)$.
(Indeed, we could then take $\varphi, \psi \equiv 1$ on $B(z_l,2\delta/3)$.)
Select $U_i$ as in Proposition \ref{ujvj}, with $B(z_l,3\delta) \subseteq U_i$.
Now, $\varphi(x)K_t(x,y)\psi(y)$ is the
kernel of the pseudodifferential operator $\varphi G(t{\sqrt \Delta})\psi$.  We can use the
coordinate map $\phi_i$ to pull this kernel over to $\RR^n$, thereby obtaining a smooth,
compactly supported kernel
$L_t$, with support in $\RR^n \times \RR^n$.
  Let us change our notation and now
use $x$ and $y$ to denote points in $\RR^n$.  By Proposition \ref{ujvj}, it is enough to show that:
\[t^{n+|\alpha|+|\beta|}\left(\frac{|x-y|}{t}\right)^N
\left|\partial_x^{\alpha} \partial_y^{\beta} L_t(x,y)\right| \leq C \]

for any multiindices $\alpha, \beta$, for all $0 < t < 1$, all $N \leq M$, and all
$x,y \in \RR^n$.

Now let $p_t(x,\xi)$ denote the symbol of the operator with kernel $L_t$.  Then
\begin{equation}\notag
L_t(x,y) =
\int e^{i(y-x)\cdot \xi} p_t(x,\xi) d\xi.
\end{equation}

Thus, $\partial_x^{\alpha} \partial_y^{\beta} L_t(x,y)$ is a finite linear combination of
terms of the form
\begin{equation}
\label{typterm}
T = \int e^{i(y-x)\cdot \xi} \xi^{\gamma} \partial_x^{\delta} p_t(x,\xi) d\xi,
\end{equation}
where $|\gamma|, |\delta| \leq |\alpha| + |\beta|$.

In case (i) we may take $M=0$,
so we need only estimate $|T|$, the absolute value of the term $T$ in
(\ref{typterm}).  It will be
enough to show that $|T| \leq Ct^{-n-|\gamma|}$
(for $0 < t < 1$), since $Ct^{-n-|\gamma|} \leq Ct^{-n-|\alpha|-|\beta|}$.  But
\begin{eqnarray*}
|T| & \leq &\int_{|\xi| \leq 1/t} |\xi^{\gamma} \partial_x^{\delta} p_t(x,\xi)|d\xi
+ \int_{|\xi| > 1/t} |\xi^{\gamma} \partial_x^{\delta} p_t(x,\xi)|d\xi \\
& &\leq C\left[A_t t^{-n-|\gamma|} + B_t \int_{1/t}^{\infty} r^{|\gamma|+n-1}r^{-|\gamma|-n-1} dr\right]\\
& &\leq C\left[A_t t^{-n-|\gamma|} + B_t t\right]
\end{eqnarray*}
where $A_t  = \sup_{x,\xi}|\partial_x^{\delta} p_t(x,\xi)|$, and
$B_t  = \sup_{x,\xi}|\xi|^{|\gamma|+n+1}|\partial_x^{\delta} p_t(x,\xi)|$.
But, by (\ref{tmdep})
and the continuity of the map $p \rightarrow p({\sqrt \Delta})$,
from $S^m_{1}(\RR)$ to $OPS^m_{1,0}({\bf M})$ in the cases $m = 0$
and $m = -(|\gamma| + n + 1)$, we find that $A_t  \leq C$ (independent of $0 < t < 1$) and
$B_t  \leq Ct^{-|\gamma|-n-1}$.  Altogether $|T| \leq Ct^{-n-|\gamma|}$, as claimed.  This
completes the proof in case (i).

In case (ii), we need only show that for every $n$-tuple $\nu$ with $|\nu|
\leq M$, we have that $|(x-y)^{\nu}T| \leq Ct^{-n-|\gamma|+|\nu|}$.
Note that
$(x-y)^{\nu} e^{i(y-x)\cdot \xi} = c \partial_{\xi}^{\nu}e^{i(y-x)\cdot \xi}$.
Substituting this in the explicit expression for $(x-y)^{\nu}T$, and repeatedly
integrating by parts in $\xi$, we see that $(x-y)^{\nu}T$ is a finite linear
combination of terms of the form
\begin{equation}\notag
T' = \int e^{i(y-x)\cdot \xi} \xi^{\kappa}
\partial_x^{\delta} \partial_{\xi}^{\chi} p_t(x,\xi) d\xi,
\end{equation}
where $|\kappa| \leq |\gamma|$, $|\chi| \leq |\nu| \leq M$,
and $|\gamma|- |\kappa| + |\chi| = |\nu|$.
Just as in our estimate for $T$ above, we see that

\begin{eqnarray*}
|T'| & \leq &
\int_{|\xi| \leq 1/t} \left|\xi^{\kappa} \partial_x^{\delta} \partial_{\xi}^{\chi} p_t(x,\xi)\right|d\xi
+ \int_{|\xi| > 1/t} \left|\xi^{\kappa} \partial_x^{\delta} \partial_{\xi}^{\chi} p_t(x,\xi)\right|d\xi \\
& &\leq C\left[A_t t^{-n-|\kappa|} + B_t \int_{1/t}^{\infty} r^{|\kappa|+n-1}r^{-|\kappa|-n-1} dr\right]\\
& &\leq C\left[A_t t^{-n-|\kappa|} + B_t t\right]
\end{eqnarray*}
where now
\[ A_t  = \sup_{x,\xi}\left|\partial_x^{\delta} \partial_{\xi}^{\chi} p_t(x,\xi)\right|
= \sup_{x,\xi}(1 + |\xi|)^{-|\chi|+|\chi|}\left|\partial_x^{\delta} \partial_{\xi}^{\chi} p_t(x,\xi)\right|,\]
and
\[B_t  = \sup_{x.\xi}|\xi|^{|\kappa|+n+1}\left|\partial_x^{\delta} \partial_{\xi}^{\chi} p_t(x,\xi)\right|
\leq
\sup_{x.\xi}(1 + |\xi|)^{|\kappa|+n+1-|\chi|+|\chi|}\left|\partial_x^{\delta} \partial_{\xi}^{\chi} p_t(x,\xi)\right|.\]
But, by (\ref{tmdep}) and the continuity of the map $p \rightarrow p({\sqrt \Delta})$,
from $S^m_{1}(\RR)$ to $OPS^m_{1,0}({\bf M})$ in the cases $m = |\chi|$
and $m = |\chi|-(|\kappa| + n + 1)$, we find that $A_t  \leq Ct^{|\chi|}$ and
$B_t  \leq Ct^{|\chi|-|\kappa|-n-1}$.  ((\ref{tmdep}) may be used here, since
$G$ vanishes to order at least $M$ at $0$, and both $|\chi|$ and
$|\chi|-(|\kappa| + n + 1)$ are less than or equal to $M$.)
Altogether $|T'| \leq Ct^{-n+|\chi|-|\kappa|} =
Ct^{-n-|\gamma|+|\nu|}$, as claimed.  This completes the proof.\end{proof}

\begin{remark} Note that, in Lemma \ref{manmol}, the conclusion (\ref{diagest}) holds even without the hypothesis
$f(0)=0$, {\em provided} $t$ is restricted to lie in the interval $(0,1]$.  Indeed, after the second sentence of
the proof of the lemma, we assumed $0 < t < 1$ and never used the hypothesis that $f(0)=0$.  Of course
(\ref{diagest}) holds also for $t=1$ by continuity.\end{remark}

\section{Continuous ${\cal S}$-Wavelets on Manifolds}\label{continuous-s-wavelets-on-manifolds}

We now turn to our definitions of continuous wavelets and continuous ${\cal S}$-wavelets on
${\bf M}$, which we have motivated in the introduction.

\begin{definition}
\label{ctswvmn}
Suppose that the function $K_t(x,y)$ is smooth for $t > 0$, $x,y \in {\bf M}$.
For $t > 0$, define $T_t: L^2({\bf M}) \rightarrow C^{\infty}({\bf M})$ to be the operator with kernel $K_t$,
so that for all $F \in L^2({\bf M})$ and all $x \in {\bf M}$,
\[ (T_t F)(x) = \int_{\bf M} K_t(x,y) F(y) d\mu(y). \]
As usual, let $P$ denote the projection in $L^2({\bf M})$ onto the space of constant functions.
Then we define $K_t(x,y)$ to be a
{\em continuous wavelet} on ${\bf M}$, provided the following three conditions hold, for some $c > 0$:
\begin{itemize}
\item[(i)] For all $F \in L^2({\bf M})$,
\begin{equation}
\label{ctswveq}
\int_0^{\infty} \|T_t F\|^2_2 \frac{dt}{t} = c \|(I-P)F\|^2_2;
\end{equation}
\item[(ii)]  $\int_{\bf M} K_t(x,y) d\mu(y) = 0$ for all $t > 0$ and all $x \in {\bf M}$ (or, equivalently, $T_t(1) = 0$ for all
$t > 0$);
\item[(iii)]  $\int_{\bf M} K_t(x,y) d\mu(x) = 0$ for all $t > 0$ and all $y \in {\bf M}$ (or, equivalently, $T_t^*(1) = 0$ for
all $t > 0$).
\end{itemize}
\end{definition}

\begin{definition}
\label{ctsSwavmn}
Suppose $K_t(x,y)$ is a continuous wavelet on ${\bf M}$.  We then say that
$K_t(x,y)$ is a continuous ${\cal S}$-wavelet on ${\bf M}$, if the following additional condition holds:
\begin{itemize}
\item[(iv)] For every pair of
$C^{\infty}$ differential operators $X$ (in $x$) and $Y$ (in $y$) on ${\bf M}$,
and for every integer $N \geq 0$, there exists $C_{N,X,Y}$ as follows.  Suppose
$\deg X = j$ and $\deg Y = k$.  Then
\begin{equation}
\label{diagest1}
t^{n+j+k} |(\frac{d(x,y)}{t})^N XYK_t(x,y)| \leq C_{N,X,Y}
\end{equation}
for all $t > 0$ and all $x,y \in {\bf M}$.
\end{itemize}
\end{definition}

We then have the following result:
\begin{theorem}
\label{ctswvthm}
Say $f_0 \in {\mathcal S}(\RR^+)$, $f_0 \not\equiv 0$, and let $f(s) = sf_0(s)$.  For $t > 0$, let $K_t$ be the kernel
of $f(t^2 \Delta)$.  Then $K_t(x,y)$ is a continuous ${\cal S}$-wavelet on ${\bf M}$.
\end{theorem}
\begin{proof} Of course, condition (iv) is Lemma \ref{manmol}.  As we have seen,
conditions (ii) and (iii) of Definitio \ref{ctswvmn} are
immediate consequences of (\ref{kerexp}), as for condition (i), say $F \in L^2({\bf M})$.
We need only take the inner product of both sides of (\ref{strint}) with $F$ to see that,
if  $c = \int_0^{\infty} |f(t)|^2 \frac{dt}{t}$, then
\[ \int_{0}^{\infty} \|f(t\Delta)F\|_2^2 \frac{dt}{t} = c\|(I-P)F\|_2^2.\]
Replacing $t$ by $t^2$ in this equation, we find that
\[ \int_{0}^{\infty} \|f(t^2\Delta)F\|_2^2 \frac{dt}{t} = \frac{c}{2}\|(I-P)F\|_2^2,\]
which yields condition (i) at once.  This completes the proof.\end{proof}

As for properties of continuous wavelets, we first remark that it is a standard, simple matter to show the following
result, which generalizes (\ref{strint}) (with $T = \Delta$ there):
\begin{proposition}
\label{ctswvinvthm}
Suppose $K_t(x,y)$ is a continuous wavelet on ${\bf M}$, and, for $t > 0$, let $T_t$ be the operator on $L^2({\bf M})$
with kernel $K_t$.  Then for any $F \in (I-P)L^2({\bf M})$, we may reconstruct $F$ through the identity
\begin{equation}
\label{ctswveq1}
\int_0^{\infty} T_t^* T_t F \frac{dt}{t} = cF.
\end{equation}
Here the integral on the left side of (\ref{ctswveq1}) conveges unconditionally in $L^2$.
\end{proposition}
\begin{proof}   Let ${\mathcal H}_1$ be the Hilbert space $(I-P)L^2({\bf M})$, and let ${\mathcal H}_2$ be the
Hilbert space $L^2(\RR^+, {\mathcal H}_1, dt/t)$.
 By our definition of continuous wavelet,
we may define  a bounded operator $U: {\mathcal H}_1 \rightarrow {\mathcal H}_2$ by
\[ UF = (T_t F)_{t > 0}. \]
Moreover, we may define a bounded operator $V: {\mathcal H}_2 \rightarrow {\mathcal H}_1$ by
\[ V(G_t)_{t > 0} = \int_0^{\infty} T_t^*G_t \frac{dt}{t} \]
where the integral converges uncondtionally in ${\mathcal H}_1$.  Indeed, let
$\|\:\|$ denote $\|\:\|_{{\mathcal H}_1}$, and let $S = \{F \in {\mathcal H}_1: \|F\| = 1\}$.
If $E \subseteq (0,\infty)$ is measurable
and contained in a compact subset of $(0,\infty)$, we have
\begin{align}\notag
\|\int_E T_t^* G_t \frac{dt}{t}\| =
 \sup_{F \in S} |\int_E  \langle  T_t^* G_t,F  \rangle  \frac{dt}{t} |
= &\sup_{F \in S} |\int_E  \langle  G_t,T_t F  \rangle  \frac{dt}{t} |\\\notag
&
\leq \left[\int_E \|G_t\|^2 \frac{dt}{t}\right]^{1/2} \left[\sup_{F \in S}|\int_E \|T_t F\|^2  \frac{dt}{t}\right]^{1/2};
\end{align}
but, by (\ref{ctswveq}), this is less than or equal to $[\int_E \|G_t\|^2 \frac{dt}{t}]^{1/2}$,
and the unconditional convergence follows.  One now readily checks that $V = U^*$.  By (\ref{ctswveq}),
$ \langle  U^*UF, F  \rangle  = c\|F\|^2$ for all $F \in {\mathcal H}_1$.  Polarizing this identity, we find
(\ref{ctswveq1}), as desired.\end{proof}

As an example of the usefulness of continuous ${\cal S}$-wavelets,
we now show the following direct analogue of a theorem of
Holschneider and Tchamitchian (\cite{hotch}):
\begin{theorem}
\label{hldchar}
Let $K_t(x,y)$ be a continuous ${\cal S}$-wavelet on ${\bf M}$, and, for $t > 0$,
let $T_t$ be the operator on $L^2$ with kernel $K_t$.
Suppose $F \in L^2({\bf M})$.  Then:\\
(a) If $F$ is H\"older continuous, with H\"older exponent $\alpha$ ($0 < \alpha \leq 1$), then for
some $C > 0$,
\begin{equation}
\label{hldcharway}
\|T_t F\| \leq C t^{\alpha}
\end{equation}
for all $t > 0$.  (Here $\|\:\|$ denotes sup norm.)\\
(b) Conversely, say $0 < \alpha < 1$, $C > 0$, and that $F$ satisfies (\ref{hldcharway})
for all $t > 0$.  Then $F$ is H\"older continuous, with H\"older exponent $\alpha$.
\end{theorem}
\begin{proof} For (a), we just note:
 \begin{align}\notag
\left| (T_t F)(x) \right|&= \left| \int
 F(y)K_t(x,y)d\mu(y)\right| \label{goholder1}\\\notag
 &=
  \left| \int \left( F(x)-F(y)\right)K_t(x,y) d\mu(y)\right|\\\notag
    &\leq Ct^{-n}\int d(x,y)^{\alpha} \left[1+d(x,y)/t\right]^{-n-1-\alpha} d\mu(y)\\\notag
     &\leq Ct^{\alpha-n}\int \left[1+d(x,y)/t\right]^{-n-1} d\mu(y)\\\notag
 &\leq Ct^{\alpha}\notag
  \end{align}
by (\ref{ptestm}), as desired.

For (b), of course $PF$, being constant, is H\"older continuous of any exponent. Since
$T_t 1= 0$, we may assume that $F = (I-P)F$.

Set $g_t = T_t F$, so that
\begin{equation}
\label{gtcta}
\|g_t\| \leq Ct^{\alpha}.
\end{equation}
For $x, y \in {\bf M}$, set $$K_{t,x}(y) =
K_t^y(x) = K_t(x,y).$$  Thus, for all $x$, $g_t(x) = \langle K_{t,x},\overline{F}\rangle$.
Since we are assuming $F \in L^2$, by Cauchy-Schwartz and (\ref{diagest1}), we obtain the
additional estimate
\begin{equation}
\label{gtctn}
\|g_t\| \leq Ct^{-n}.
\end{equation}
By (\ref{diagest1}), we have that $|K_t(x,y)| \leq Ct^{-n}(1 + d(x,y)/t)^{-n-1}$, so by
(\ref{ptestm}),
\begin{equation}
\label{kty1}
\|K_t^y\|_1 \leq C
\end{equation}
for any $y$.  Thus, for any $y$,
\begin{equation}\notag
|(T_t^* T_t F)(y)| = |\langle g_t, K_t^y \rangle| \leq \|g_t\|\:\|K_t^y\|_1 \leq C\|g_t\|.
\end{equation}
Accordingly, by (\ref{gtcta}) and (\ref{gtctn}), for any $y$,
\begin{equation}\notag
\int_0^{\infty} |(T_t^* T_t F)(y)| \frac{dt}{t} \leq C\left(\int_0^1 t^{\alpha - 1} dt +
\int_1^{\infty} t^{-n - 1} dt\right)  \leq C.
\end{equation}
By Proposition \ref{ctswvinvthm}, we now see that for almost every $y$,
\[ cF(y) = \int_0^{\infty} T_t^* T_t F(y) \frac{dt}{t}, \]
and from this, that $F \in L^{\infty}$.

To complete the proof, we claim that it suffices to show that if
$d(y,z) \leq \min(t,\delta)$, then
\begin{equation}
\label{ktxyz}
\int_{\bf M} |K_t(x,y)-K_t(x,z)| d\mu(x) \leq Cd(y,z)/t.
\end{equation}
For then, by (\ref{gtcta}), (\ref{gtctn}), (\ref{kty1}) and (\ref{ktxyz}), we would have,
if $d(y,z) < \delta$, then
\begin{eqnarray*}
c|F(y)-F(z)| & \leq & \int_0^{\infty} \int_{\bf M} |K_t(x,y)-K_t(x,z)|d\mu(x) \|g_t\| \frac{dt}{t}\\
& \leq & C\left[ \int_0^{d(y,z)} t^{\alpha-1} dt + d(y,z)\int_{d(y,z)}^1 t^{\alpha-2} dt
+ d(y,z)\int_1^{\infty} t^{-n-2} dt\right]\\
& \leq & C d(y,z)^{\alpha},
\end{eqnarray*}
as needed.

To prove (\ref{ktxyz}), choose
$U_i \supseteq B(y,3\delta)$, and let us work in the local coordinates on $U_i$ obtained
from $\phi_i$.  We use the mean value theorem.
By (\ref{diagest1}), for any $x \in {\bf M}$, there is
point $w_x$ on the line segment joining $y$ to $z$ such that
\[ |K_t(x,y)-K_t(x,z)| \leq C|y-z| t^{-n-1}(1 + d(x,w_x)/t)^{-n-1}.\]
By Proposition \ref{ujvj}, if $w$ is any point on that line segment,
\[ d(y,w) \leq c_2|y-w| \leq c_2|y-z| \leq c_1c_2d(y,z) \leq c_1c_2t. \]
Thus the diameter of the line segment is at most $2c_1c_2t$, and so, by (\ref{alcmpN}), we have
\[ |K_t(x,y)-K_t(x,z)| \leq Cd(y,z) t^{-n-1}(1 + d(x,y)/t)^{-n-1}.\]
(\ref{ktxyz}) now follows from (\ref{ptestm}), as desired.\end{proof}

\section{Homogeneous Manifolds}\label{homogeneous-manifolds}
In this section look at the situation in which ${\bf M}$ has a
transitive group $G$ of smooth metric isometries.  (Such manifolds are usually called
homogneous manifolds.)  Obvious examples of such manifolds are the
sphere $S^n$, where we take $G$ to be the group $SO(n+1)$ of rotations, and the torus
${\bf T}^n = (S^1)^n$, where we take $G$ to be the group $[SO(2)]^n$.

If $T \in G$ and $F$ is a function on ${\bf M}$, we define the function $TF$ on ${\bf M}$ by
$(TF)(x) = F(T^{-1}x)$.   Then $T: L^2({\bf M}) \rightarrow L^2({\bf M})$ is a unitary operator
which commutes with the Laplace-Beltrami operator $\Delta$.

Consequently, as operators on $L^2({\bf M})$,
$f(t^2\Delta)$ commutes with elements of $G$ for any bounded Borel function $f$ on $\RR$,
and in particular, if $f \in {\mathcal S}(\RR)$, which we now assume.

Thus, if $T \in G$, $F \in L^2({\bf M})$ and $x\in {\bf M}$, we have
\begin{align} \notag
\int_{\bf M} K_t(Tx, Ty) F(y) d\mu(y) &= \int_{\bf M} K_t(Tx, y) F(T^{-1}y) d\mu(y) \\\notag &= [f(t^2\Delta)(TF)](Tx)\\\notag
&= T([f(t^2\Delta)(F)])(Tx);
\end{align}
but this is just $[f(t^2\Delta)(F)](x) = \int_{\bf M} K_t(x, y) F(y) d\mu(y)$, so
\begin{equation}
\label{rotinv}
K_t(Tx,Ty) = K_t(x,y)
\end{equation}
for all $x,y \in {\bf M}$.
Thus, since $G$ is transitive, if $x_0$ is any fixed point in ${\bf M}$, once one knows
$K_t(x_0,y)$ for all $y$, then one knows $K_t(x,y)$ for all $x,y$.

In the analogous situation on $\RR^n$, $K_t(x,y)$ has the form
$t^{-n} \psi((x-y)/t)$ for some $\psi \in \mathcal{S}$, and so $K_t(x,y) =
K_t(Tx,Ty)$ for any {\em translation} $T$ on $\RR^n$.  Equation (\ref{rotinv}) is a natural
analogue of this fact for ${\bf M}$.

It is interesting to note that one has
\begin{equation}
\label{indlink1}
K_t(x,x) =  \mbox{tr}(f(t^2\Delta))/\mbox{vol}({\bf M})
\end{equation}
for all $x$ and all $f \in {\mathcal S}(\RR^+)$.
Indeed, by (\ref{rotinv}), $K_t(x,x)$ is constant for $x \in {\bf M}$.  Accordingly
\[ \mbox{vol}({\bf M}) K_t(x,x) = \int_{\bf M} K_t(y,y) d\mu(y)
= \sum_l f(t^2\lambda_l)\int_{\bf M}|u_l(y)|^2 d\mu(y)
= \sum_l f(t^2\lambda_l) = \mbox{tr}(f(t^2\Delta))\]
as claimed.

Say now $c > 0$, and let us look at the
special case
\begin{equation}
\label{mxstup}
f(s) = (s/c)e^{-(s/c)}.
\end{equation}
We have $\mbox{tr}(f(t^2\Delta))
= \mbox{tr}((t^2/c)\Delta e^{-(t^2\Delta)/c})$.
A well known fact, usually associated with the
heat kernel approach to index theorems (\cite{MinPle}, \cite{Gilkey},
\cite{Polt} and \cite{Gilkeybook}, pages 58 and 316), is that as $s \rightarrow 0^+$,
\begin{equation}
\label{heattras}
\mbox{tr}(e^{-s\Delta}) \sim \sum_{m=0}^{\infty} s^{m-n/2}a_m,
\end{equation}
where
\begin{equation}
\label{heattr0}
a_0 = (4\pi)^{-n/2} \mbox{vol}({\bf M}).
\end{equation}
Differentiating with respect to $s$ one finds that
\[ \mbox{tr}(\Delta e^{-s\Delta}) \sim \sum_{m=0}^{\infty} (\frac{n}{2}-m)s^{m-1-n/2}a_m. \]
To lowest order, then,
\begin{equation}
\label{mexhtxx}
 K_t(x,x) =
\mbox{tr}((t^2/c)\Delta e^{-(t^2\Delta)/c})/\mbox{vol}({\bf M}) \sim \frac{nc t^{-n}}{2(4\pi)^{n/2}}.
\end{equation}
\ \\
Again let $f$ be general, but now let us look at the special case ${\bf M} = {\bf T}^n$, the torus.
We write ${\bf T}^n = \left\{(e^{2 \pi i r_1},\ldots, e^{2 \pi i r_n}):\; -1/2 < r_1, \ldots, r_n \leq 1/2\right\}$.
Here an orthonormal basis of eigenfunctions is given simply by $\{ e^{2\pi i m \cdot r}: m \in {\ZZ}^n \}$.  The
Laplace-Beltrami operator $\Delta$ is just $-\sum_{l = 1}^n (\partial/ \partial r_l)^2$, and the
eigenvalues are given through $\Delta e^{2\pi i m \cdot r} = 4 \pi^2 \|m\|^2 e^{2\pi i m \cdot r}$.
(Here $\|m\|^2 = m_1^2 + \ldots m_n^2$.)  The kernel $K_t(r,s)$ of $f(t^2\Delta)$ is given by
\[ K_t(r,s) = \sum_{m \in \ZZ^n} f(4\pi^2 t^2 \|m\|^2) e^{2\pi i m \cdot (r-s)}. \]
Thus, if $F \in L^2({\bf T}^n)$,
\[ [f(t^2 \Delta)F](r) = \int_{-1/2}^{1/2} \ldots \int_{-1/2}^{1/2} F(s) K_t(r,s) ds_1\ldots ds_n = [F*h_t(r)] \]
where
\[ h_t(s) = \sum_{m \in (\ZZ)^n} f(4\pi^2 t^2 \|m\|^2) e^{2\pi i m \cdot s}, \]
and $*$ denotes the natural convolution on ${\bf T}^n$.

We specialize now to the case $n = 2$, $f(u) = ue^{-u/4\pi}/4\pi^2$.  We free the letter $n$ for other uses.
For $t > 0$ define the functions $U_t, V_t : \RR \rightarrow \CC$ by
\begin{equation}
\label{utdf}
U_t(x) = \sum_{n=-\infty}^{\infty} e^{-\pi t^2 n^2} e^{2\pi i n x},
\end{equation}
\begin{equation}
\label{vtdf}
V_t(y) = \sum_{n=-\infty}^{\infty}(nt)^2 e^{-\pi t^2 n^2} e^{2\pi i n y}.
\end{equation}
It is then easy to calculate that
\begin{equation}\notag
h_t(s_1,s_2) = U_t(s_1) V_t(s_2) + U_t(s_2) V_t(s_1).
\end{equation}
$h_t$ is the ``Mexican hat" on the torus ${\bf T}^2$.  One can use these equations to draw its graph.  One
of course needs to approximate the series in (\ref{utdf}) and (\ref{vtdf}) by finite sums.  This is a simple
matter if $t$ is greater than $1$, but if $t$ is small the series do not converge very quickly.  Fortunately,
however, we can give alternative series expansions for $U_t$ and $V_t$
which do converge very quickly for $0 < t < 1$.   We do this by using the Poisson summation formula, or rather,
its proof: if $g \in {\mathcal S}(\RR)$, then the periodic function $G(x) = \sum_{n = -\infty}^{\infty} g(x + n)$
has Fourier series $\sum_{n = -\infty} \check{g}(n) e^{2\pi i nx}$, and hence, $G(x)$ equals the latter series.
(Here we use the inverse Fourier transform $\check{g} (x) = \int_{-\infty}^{\infty} g(\xi) e^{-2\pi i x \xi} d\xi$.)
Applying this with $\check{g}(y) = e^{-\pi t^2 y^2}$, we obtain the formula
\begin{equation}
\label{utalt}
U_t(x) = \frac{1}{t} \sum_{n=-\infty}^{\infty} e^{-\pi (n+x)^2/t^2}.
\end{equation}
Taking instead $\check{g}(y) = (ty)^2 e^{-\pi t^2 y^2}$, we obtain the formula
\begin{equation}
\label{vtalt}
V_t(x) = \frac{1}{t} \sum_{n=-\infty}^{\infty}(\frac{1}{2\pi} - (\frac{n+x}{t})^2) e^{-\pi (n+x)^2/t^2}.
\end{equation}
(\ref{utalt}) and (\ref{vtalt}) converge very quickly for $t$ small, making it practical to draw
pictures of $h_t$ for $t$ small.  We include
pictures, obtained by using Maple, of the Mexican hat
functions $\pi h_t(r_1,r_2)$ $-1/2 < r_1, r_2 \leq 1/2$,
for $t = 2$ (Figure 1, left), $t = 1/2$ (Figure 1, middle) and $t = 1/8$
(Figure 1, right).
\begin{figure}
\begin{center}
\begin{tabular}{ccc}
  \includegraphics[scale=0.32]{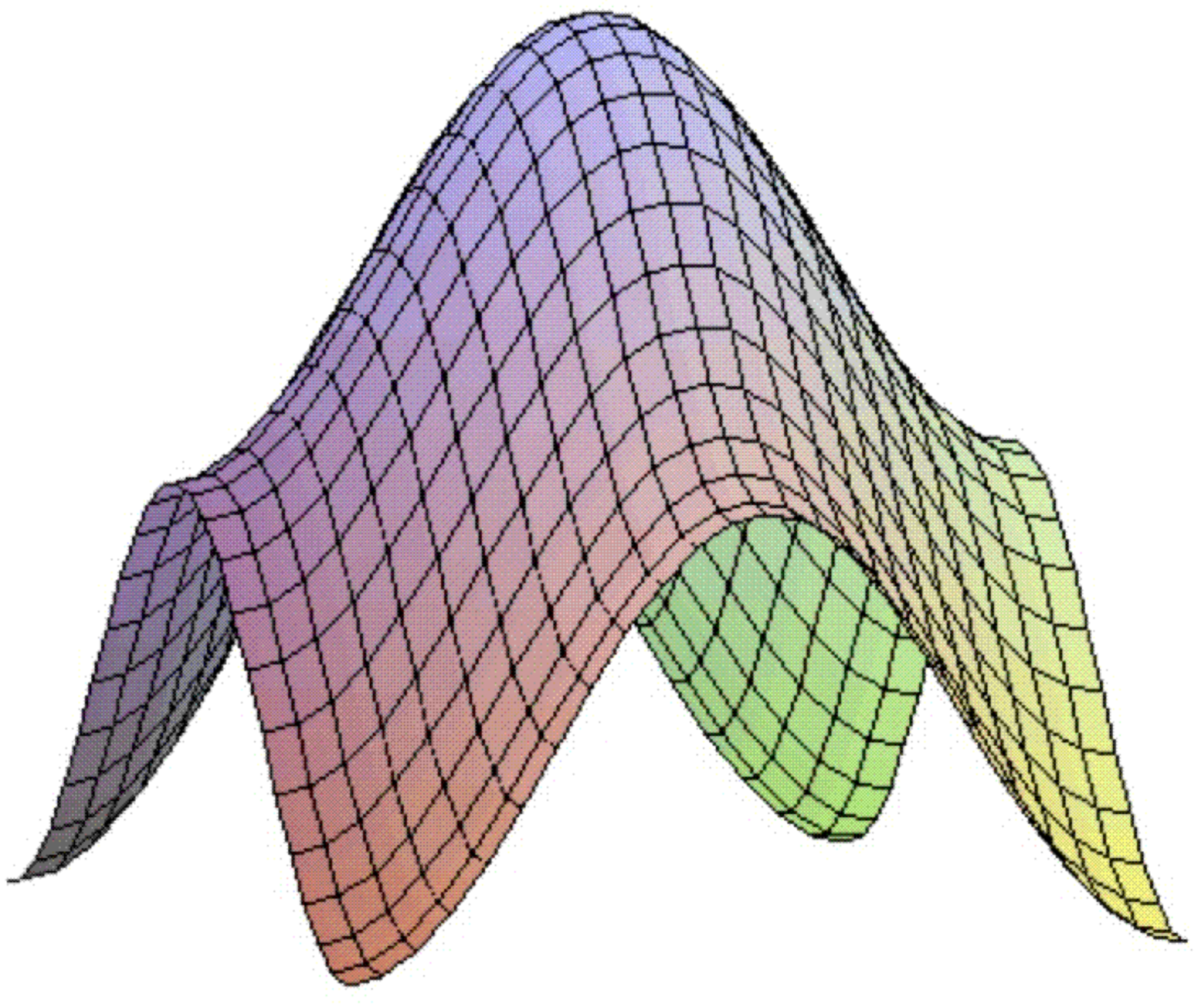}  &
      \includegraphics[scale=0.32]{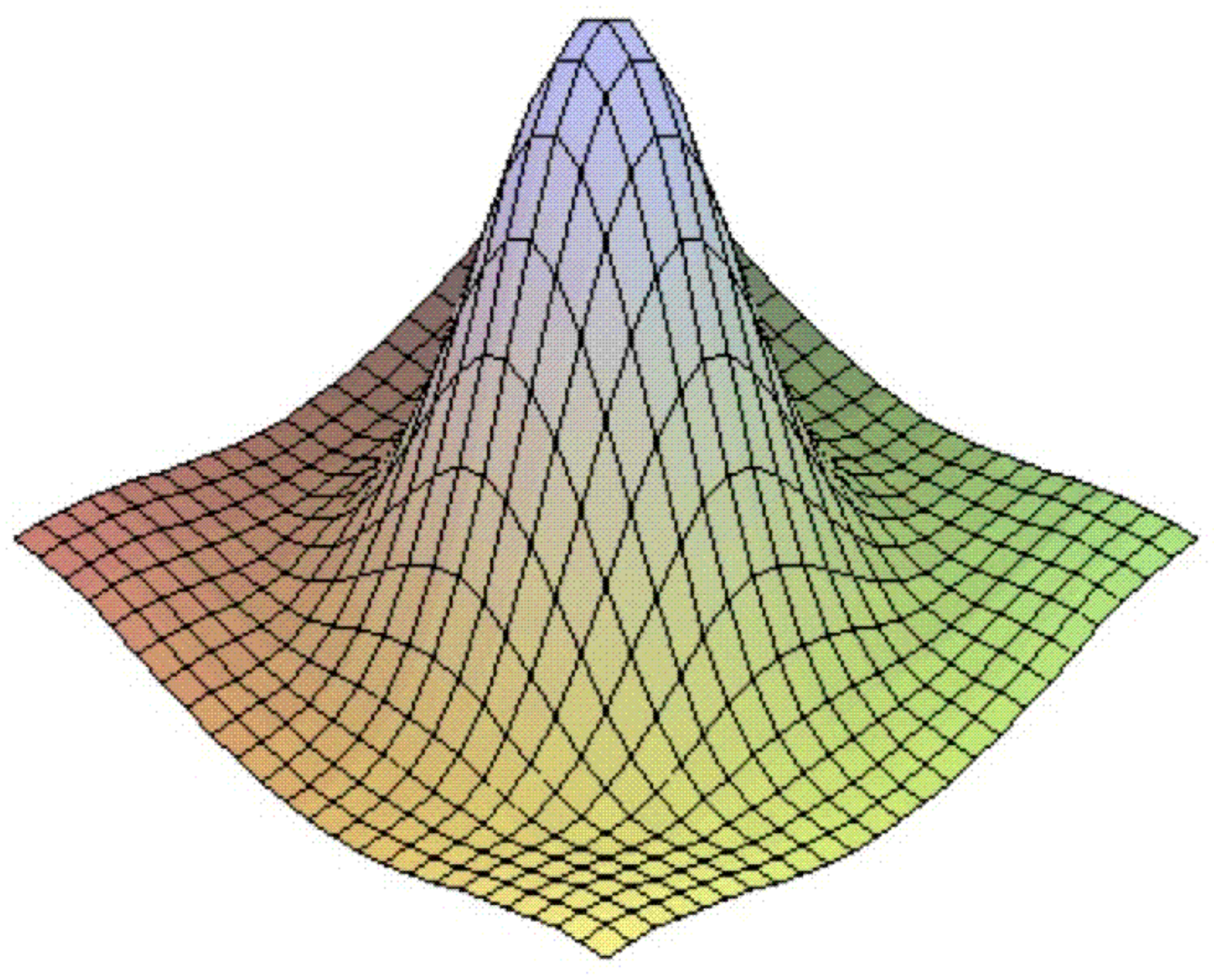}&
      \includegraphics[scale=0.32]{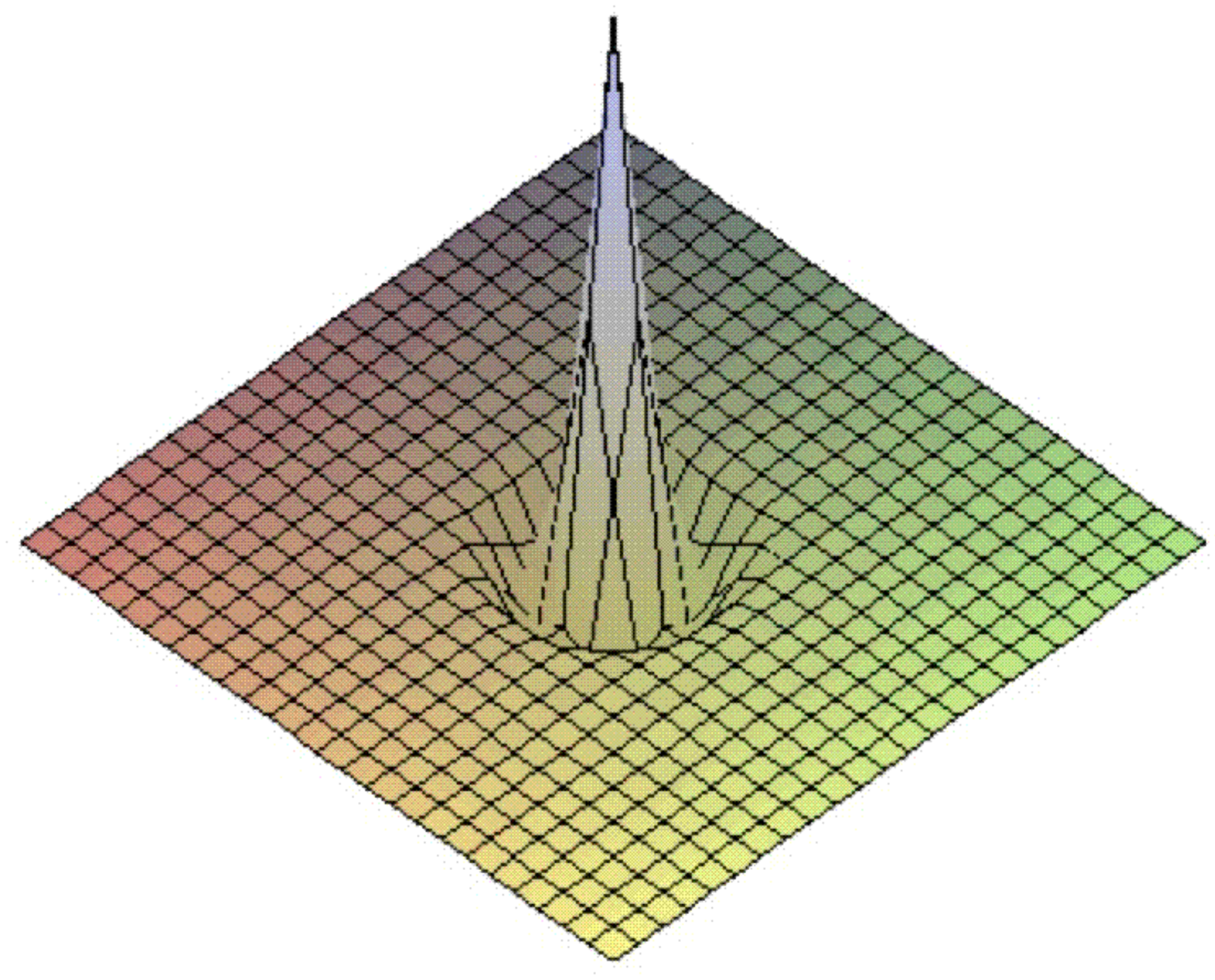}\\
         \end{tabular}\\
        \caption{\label{fig1.pdf}\small  $h_t$ on ${\bf T}^2$ for $t=2$ (left),  $t=1/2$ (middle),
     $t=1/8$ (right)}
  \end{center}
  \end{figure}
Note that the characteristic Mexican hat shape is obtained.  To keep the pictures uncluttered,
we have omitted the axes.  However, one can ask Maple to evaluate $t^2\pi h_t(0,0)$ for various
values of $t$.  When $t=2$, it is $.00070$; when
$t=1$, it is $.59017$; when $t=1/2$, it is $.99984$, and when $t=1/8$,
it is $1.00000$.  This is consistent with the predictions of
(\ref{mxstup}) with $c = 4\pi$, and (\ref{mexhtxx}).\\

Now let us look at the special case ${\bf M} = S^n$, and let $f$ be general for now.
Let ${\bf N} = (1,0,\ldots,0)$, the
``north pole".  We shall now
calculate $K_t({\bf N},y)$ as an explicit infinite series.  (As we have
explained, we will then know $K_t(x,y)$ for all $x,y$.)
Recall (\cite{StWeis71}) that we may write $L^2(S^n) = \bigoplus_{l \geq 0}{\mathcal H}_l$,
where ${\mathcal H}_l$ is the space of spherical harmonics of degree $l$.  If $P \in
{\mathcal H}_l$, then
\[ \Delta P = l(l+n-1) P. \]
Within each space ${\mathcal H}_l$ is a unique {\em zonal harmonic} $Z_l$, which has
the property that for all $P \in {\mathcal H}_l$, $P({\bf N}) =  \langle  P,Z_l  \rangle $.
In particular, $P$ is orthogonal to $Z_l$ if and only if $P({\bf N}) = 0$.

We wish to use (\ref{kerexp}) to evaluate $K_t({\bf N},y)$.  To do so we
choose an orthonormal basis for each ${\mathcal H}_l$, one of whose elements is
$Z_l/\|Z_l\|_2$.
Then any other element of this orthonormal basis vanishes at ${\bf N}$, so we find

\begin{equation}\notag
K_t({\bf N},y) = \sum_{l=0}^{\infty} f(t^2l(l+n-1)) Z_l({\bf N})Z_l(y)/\|Z_l\|_2^2.
\end{equation}

But surely $Z_l({\bf N}) =  \langle  Z_l,Z_l  \rangle $, so we simply have

\begin{equation}
\label{kersph2}
K_t({\bf N},y) = \sum_{l=0}^{\infty} f(t^2l(l+n-1))Z_l(y).
\end{equation}

However, $Z_l(y)$ is known explicitly.  In fact (\cite{StWeis71}),
if $\omega_n$ is the area of $S^n$, then for some constant $c_l$,

\begin{equation}
\label{zony}
Z_l(y) = c_l P_l^{\lambda}(y_1),
\end{equation}
where $y = (y_1,\ldots,y_n)$, $\lambda = (n-1)/2$, and
$P^{\lambda}_l$ is the ultraspherical (or Gegenbauer)
polynomial of degree $l$ associated with $\lambda$.
The $P^{\lambda}_l$
may be defined in terms of the generating function
\begin{equation}
\label{genfn}
(1-2r\tau+r^2)^{-\lambda} = \sum_{l=0}^{\infty} P_l^{\lambda}(\tau)r^l.
\end{equation}
In particular, if $\tau = 1$, we see that
\[ (1-r)^{-(n-1)} = \sum_{l=0}^{\infty} P_l^{\lambda}(1)r^l, \]
so that

\begin{equation}
\label{pk1}
P_l^{\lambda}(1) = \left(\displaystyle^{n+l-2}_{\ \ \ l\ \ \ }\right) = b_l
\end{equation}
On the other hand,
\begin{equation}
\label{zonn}
Z_l({\bf N}) = [\omega_n]^{-1}\dim {\mathcal H}_l
= [\omega_n]^{-1}\left[\left(\displaystyle^{n+l}_{\ \:n\ }\right)-\left(\displaystyle^{n+l-2}_{\ \ \ l\ \ \ }\right)\right] = d_l.
\end{equation}

Comparing (\ref{zony}),(\ref{pk1}) and (\ref{zonn}), we see that
\begin{equation}
\label{ckeval}
c_l = d_l/b_l = \frac{\left(\displaystyle^{n+l}_{\ \:n\ }\right)-\left(\displaystyle^{n+l-2}_{\ \ \ n\ \ \ }\right)}
{\omega_n\left(\displaystyle^{n+l-2}_{\ \ \ l\ \ \ }\right)} = \frac{n+2l-1}{\omega_n(n-1)}.
\end{equation}

 From (\ref{kersph2}), we find
\begin{equation}
\label{kersph3}
K_t({\bf N},y) = \sum_{l=0}^{\infty}
\frac{(n+2l-1)}{\omega_n(n-1)}f(t^2l(l+n-1)) P_l^{\lambda}(y_1) := h_t(y_1).
\end{equation}
If $x \in S^n$, we can choose a rotation $T$ with $Tx = {\bf N}$.  If also $y \in
S^n$, then by (\ref{rotinv}),
\[ K_t(x,y) = K_t(Tx,Ty) = K_t({\bf N},Ty) = h_t((Ty)_1) = h_t({\bf N}\cdot Ty)
= h_t(Tx \cdot Ty) = h_t(x \cdot y). \]

Thus, if $F \in L^2(S^n)$,
\[ [f(t^2 \Delta) F](x) = \int_{S^n} F(y) h_t(x \cdot y) d\mu(y), \]
the {\em spherical convolution} of $F$ and the axisymmetric function $h_t$.

Let us now take $f(s) = se^{-s}$.
 From (\ref{kersph3}), we find
\begin{equation}
\label{kersph4}
h_t(y_1) = K_t({\bf N},y) = \sum_{l=0}^{\infty}
\frac{l(l+n-1)(n+2l-1)}{\omega_n(n-1)}t^2e^{-t^2l(l+n-1)} P_l^{\lambda}(y_1).
\end{equation}
 One can use (\ref{kersph4}) to draw the graph of $h_t$ for any
$t > 0$.  We do so, in the most practical situation, $n = 2$.
We can go all around a great circle by using spherical coordinates, in
which $y_1 = \cos(\theta)$, with $\theta$ going from $-\pi$ to $\pi$.
We draw these graphs when $t = 1$ (Figure 2, left), $t = .1$ (Figure 2, middle)
 and $t = .05$ (Figure 2, right).  Actually, since (\ref{mexhtxx}) predicts
\[ 4\pi h_t(1)  = 4\pi K_t({\bf N}, {\bf N}) \sim 1/t^2,  \]
we draw the graphs of
$4\pi\:h_t(\cos(\theta))$ instead, with $\theta$ going from $-\pi$ to $\pi$
on the horizontal axis.
\begin{figure}
  \begin{center}
  \begin{tabular}{ccc}
   \includegraphics[scale=0.30]{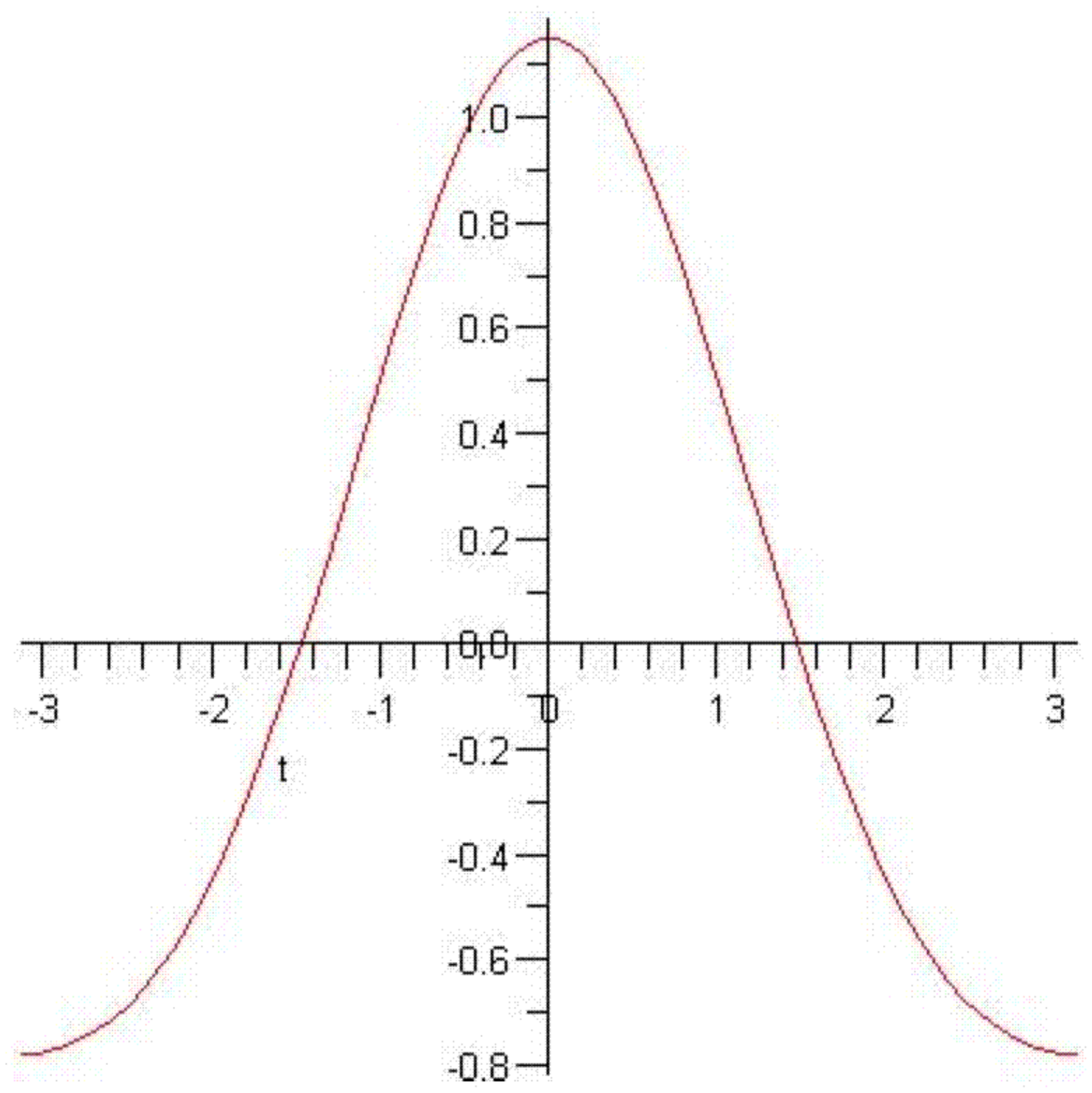}
   \includegraphics[scale=0.35]{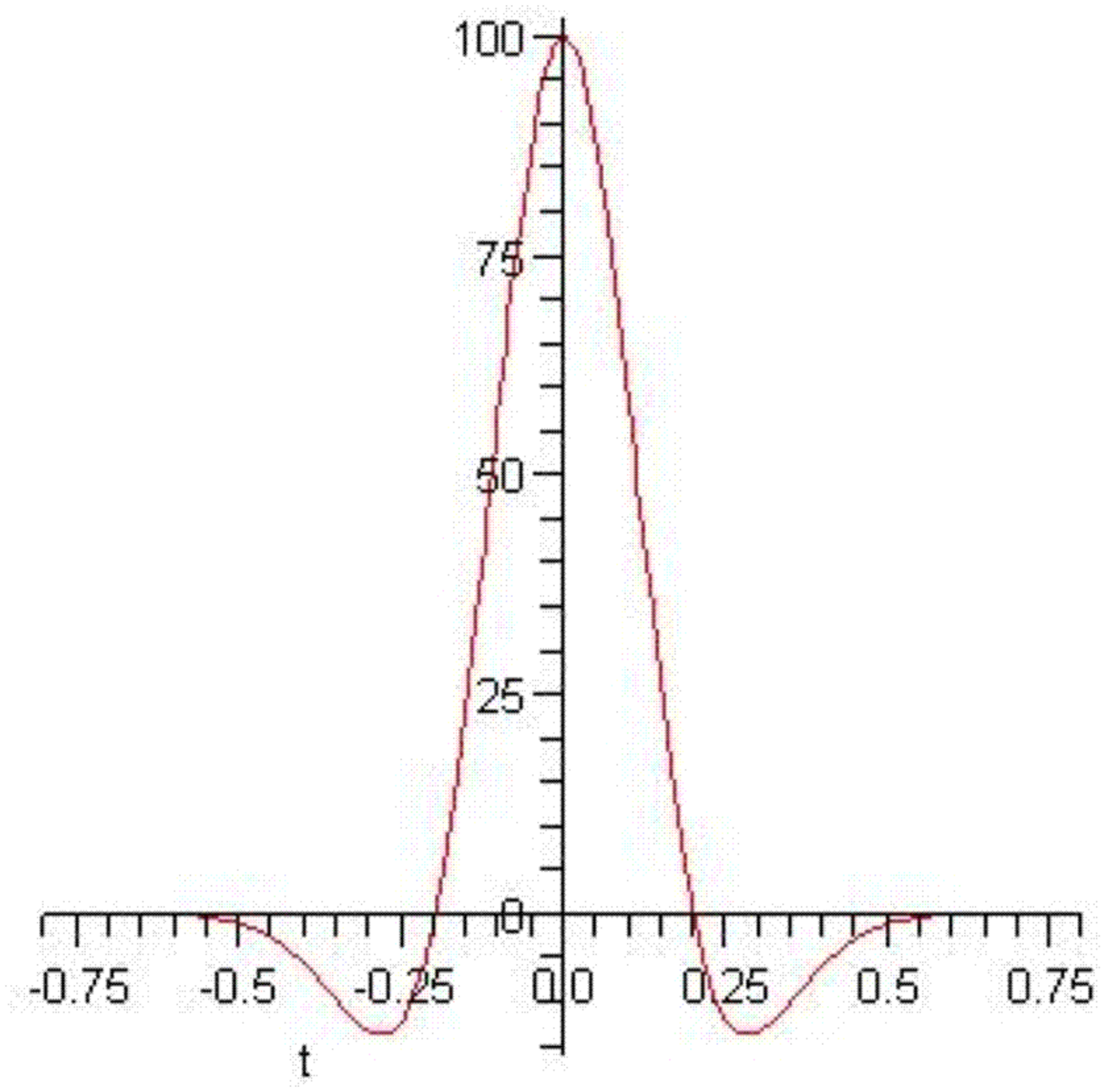}
     \includegraphics[scale=0.35]{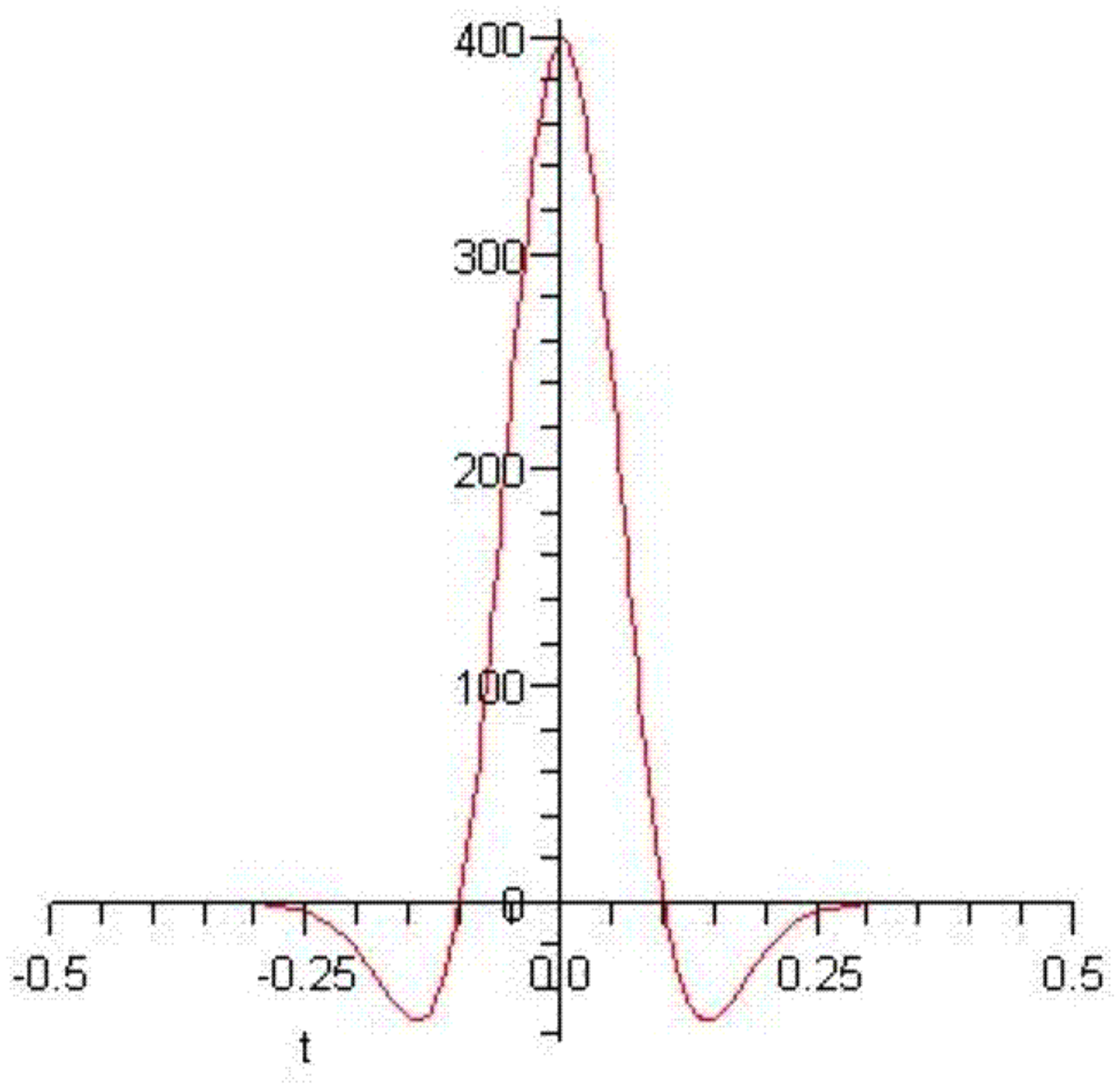}
\end{tabular}
   \caption{ \label{fig2.pdf} $4\pi h_t(\cos \theta)$ on $S^2$ for $t=1$ (left),
  $t=.1$ (middle),
  $t=0.05$ (right) }
 \end{center}
  \end{figure}
   The pictures do bear out the relation $4 \pi h_t(1) \sim 1/t^2$.  We also see the
characteristic ``Mexican hat'' shape, familiar from the analogous situation on $\RR^1$,
where $K_t(x,y)$ is the Schwartz kernel of $f(-t^2 d^2/dx^2)$.  Of course, in that situation,
$\left[f(-t^2 d^2/dx^2)F\right]\hat{\:}(\xi) = f(t^2 \xi^2) \hat{F}(\xi) =
t^2 \xi^2 e^{-t^2 \xi^2} \hat{F}(\xi)$ for $F \in {\mathcal S}(\RR)$,
so $K_t(x,y) = t^{-1} \psi((x-y)/t)$, where $\psi$ is the second derivative of a
Gaussian.

The series (\ref{kersph4}) converges quickly for $t$ large, but not if $t$ is small.  Therefore,
as in the case of the torus, for computational purposes, it is important to have a quickly
converging alternate expression for this function for small $t$.  Since the pictures indicate
that $4\pi h_t(\cos \theta)$ is negligible outside a small neighborhood of $\theta = 0$ if $t$ is small, one would
assume that it is only necessary to compute the Maclaurin series of $h_t$.  It is very fortunate that
any number of terms of this series can be computed explicitly,, through use of the work of
Polterovich \cite{Psph}, together with some additional insights.

In fact, what Polterovich found in \cite{Psph} was the heat trace asymptotics on the sphere, i.e.
explicit formulae for the $a_m$ of ({\ref{heattras}) for the manifold $S^n$, as sums of explicit
finite series, for $m \geq 1$.  (Earlier, less explicit formulae were
found earlier in (\cite{CahnWolf}) and (\cite{Camp}).)
Using Polterovich's formula to evaluate some of these $a_m$ on $S^2$, (it is easiest to use Maple), we find that
\begin{equation}
\label{heattrasx}
\mbox{tr}(e^{-s\Delta}) \sim \frac{1}{s} + \frac{1}{3} + \frac{s}{15} + \frac{4s^2}{315} +
\frac{s^3}{315} + O(s^4);
\end{equation}
there would be no difficulty in evaluating more terms.  (Recall that $a_0$ is given by
(\ref{heattr0}), for general ${\bf M}$.)

Using this formula we are going to evaluate the first few terms of the Maclaurin series of
$4\pi h_t(\cos \theta)$, and
we will show how any number of terms could be obtained.

Let $J_t(x,y)$ be the kernel of $e^{-t^2\Delta}$ on $S^2$, so that, in particular, by
(\ref{indlink1}),  $J_t(x,x) =  \mbox{tr}(e^{-t^2\Delta})/4\pi$.
From (\ref{kersph3}), with $f(r) = e^{-r}$, we find that if $s = t^2$, then
\begin{equation}
\label{kersph5}
J_t({\bf N},y) = \sum_{l=0}^{\infty} \frac{(2l+1)}{4\pi}e^{-sl(l+1)} P_l^{\lambda}(y_1) := g_t(y_1).
\end{equation}
and similarly, from (\ref{kersph4}),
\begin{equation}
\label{kersph6}
\frac{1}{s}K_t({\bf N},y) = \sum_{l=0}^{\infty}
\frac{l(l+1)(2l+1)}{4\pi}e^{-sl(l+1)} P_l^{\lambda}(y_1)=\frac{1}{s}h_t(y_1).
\end{equation}

We would like to understand the Maclaurin series of $4\pi g_t(\cos \theta)$ and $4\pi h_t(\cos \theta)$.
To this end, we shall use the following lemma.
\begin{lemma}
\label{sphlpx}
Suppose $U \in C^2(S^{n})$ is a function of $x_1$ only, $U(x) = u(x_1)$, say,
and that $u$ is in fact $C^2$ in a neighborhood of $[-1,1]$.
Then $\Delta U$ is also a function of $x_1$, and
\begin{equation}
\label{lapgx1}
(\Delta U)(x) = nx_1 u'(x_1) - (1-x_1^2)u''(x_1).
\end{equation}
In particular, if $n=1$, one has
\begin{equation}
\label{lapgx2}
-\frac{d^2}{d\theta^2} [u(\cos \theta)] = (\cos \theta) u'(\cos \theta) - (\sin^2 \theta)u''(\cos \theta).
\end{equation}
\end{lemma}
\begin{proof} By the results of \cite{gesph}, $\Delta = - \sum_{j < k}W_{jk}^2$, where
\[ W_{jk} = x_j\frac{\partial}{\partial x_k} - x_k\frac{\partial}{\partial x_j}. \]
if $1 \leq j,\:k \leq n+1$.
Since $U$ is a function of $x_1$ only we compute
\begin{eqnarray*}
\Delta U & = & - \sum_{k=1}^{n+1} W_{1k}^2 U \\
& = & -\sum_{j=2}^{n+1} [-x_1 u' + x_j^2 u'']
\end{eqnarray*}
which proves (\ref{lapgx1}), since on the unit sphere, $\sum_{j=2}^{n+1} x_j^2 = 1-x_1^2$.
Of course (\ref{lapgx2}) is elementary, but we note that it can also be viewed as a special case of
(\ref{lapgx1}), since, in polar coordinates, the spherical Laplacian on $S^1$ is just
$-\frac{d^2}{d\theta^2}$.   This completes the proof.\end{proof}

In particular, if we set $x_1 = 1$ in (\ref{lapgx1}), and $\theta = 0$ in (\ref{lapgx2}),
we find that
\begin{equation}
\label{d2gth}
-\frac{d^2}{d\theta^2} [u(\cos \theta)]|_{\theta = 0} = \frac{1}{n}(\Delta U)({\bf N}),
\end{equation}
since both sides equal $u'(1)$.
\begin{remark} We shall show below, in Lemma \ref{sphlpall}, that if $u$ is $C^{2m}$ on an
open interval containing $[-1,1]$,  one can similarly obtain
$\frac{d^{2m}}{d\theta^{2m}} [u(\cos \theta)]|_{\theta = 0}$ from a knowledge of
$(\Delta^i U)({\bf N})$ for $i = 1,\ldots,m$.  Thus,
if $u$ is $C^{\infty}$ on an
open interval containing $[-1,1]$,  the entire Maclaurin series of
$u(\cos \theta)$ (regarded as a function of $\theta$) is completely determined
from a knowledge of $(\Delta^i U)({\bf N})$ for $i \geq 0$. (Of course, $u(\cos \theta)$ is an
even function of $\theta$, so all of its odd-order derivatives vanish at $0$.)\end{remark}

For now, let us apply this lemma to $u=g_t$  as in (\ref{kersph5}) or $\frac{1}{s}h_t$ as in
(\ref{kersph6}).  To do so, we first explain why
$g_t$ and $h_t$ have $C^{\infty}$ extensions to an open neighborhood of $[-1,1]$.  For this, it is evidently
sufficient to show that for every $n \geq 1$ and every $m \geq 0$, there exists $c_{n,m}$ with
\[ \left|\frac{d^m}{d\tau^m}P_l^{(n-1)/2}(\tau)\right| \leq c_{n,m} l^{n+2m-1}. \]
From the generating function formula (\ref{genfn}), we see the classical formula that
the derivative of $P_l^{\lambda}$ is $2\lambda P_{l-1}^{\lambda+1}$.  Thus we may assume $m=0$.
By (\ref{zony}) and (\ref{ckeval}), we need only show that, on $S^n$, $|Z_l(y)| \leq C l^n$, and for this, by
(\ref{zonn}), we need only show that $|Z_l(y)| \leq Z_l({\bf N})$ for all $y \in S^n$.
But we may choose an orthogonal transformation $T$ on $\RR^n$ with $T{\bf N} = y$, and then
\begin{equation}\notag
Z_l(y) = Z_l \circ T({\bf N}) = \langle Z_l \circ T, Z_l \rangle \leq \|Z_l \circ T\|_2 \|Z_l\|_2
= \|Z_l\|_2^2  = Z_l({\bf N}),
\end{equation}
as claimed.

Let us return to $S^2$.  Suppose as usual that $t > 0$, and again put $s = t^2$.  By (\ref{d2gth}) and
(\ref{kersph5}) we have that
\begin{align}\notag
-\frac{d^2}{d\theta^2} 4\pi[g_t(\cos \theta)]|_{\theta = 0} &= \frac{4\pi}{2}(\Delta_y J_{\sqrt s})({\bf N},{\bf N})
\\\notag
&= -\frac{4\pi}{2} \frac{d}{ds} J_{\sqrt s}({\bf N},{\bf N})\\\notag
& =
-\frac{1}{2}\frac{d}{ds} \mbox{tr}(e^{-s\Delta}).
\end{align}
Thus, the first few terms of the Maclaurin series of $4\pi[g_t(\cos \theta)]$ are
\begin{equation}\notag
4\pi g_t(\cos \theta) \sim \mbox{tr}(e^{-s\Delta}) + \frac{\theta^2}{4}\frac{d}{ds} \mbox{tr}(e^{-s\Delta}).
\end{equation}
Similarly
\begin{align}\notag
-\frac{d^2}{d\theta^2} 4\pi[\frac{1}{s}h_t(\cos \theta)]|_{\theta = 0}
&= \frac{4\pi}{2}(\Delta_y \frac{1}{s}K_{\sqrt s})({\bf N},{\bf N})\\\notag
&= -\frac{4\pi}{2} \frac{d}{ds} \frac{1}{s}K_{\sqrt s}({\bf N},{\bf N})\\\notag
& =
+\frac{1}{2}\frac{d^2}{ds^2} \mbox{tr}(e^{-s\Delta}).
\end{align}
Thus, the first few terms of the Maclaurin series of $4\pi[g_t(\cos \theta)]$ and $4\pi[h_t(\cos \theta)]$ are
\begin{equation}
\label{gtmac2}
4\pi g_t(\cos \theta) \sim \mbox{tr}(e^{-s\Delta}) + \frac{\theta^2}{4}\frac{d}{ds} \mbox{tr}(e^{-s\Delta}).
\end{equation}
\begin{equation}
\label{htmac1}
4\pi h_t(\cos \theta) \sim -s\frac{d}{ds}\mbox{tr}(e^{-s\Delta})
-\frac{\theta^2}{4}s\frac{d^2}{ds^2} \mbox{tr}(e^{-s\Delta}).
\end{equation}
In order to put these into a useful form, we now invoke a result of Kannai \cite{K}.
\newline Set $z(\theta) =
(\cos\theta,\sin\theta,0,\ldots,0) \in S^n$.
Since $g_t(\cos \theta)$ equals $J_t({\bf N},z(\theta))$, where $J_t(x,y)$ is the kernel of the
heat operator $e^{-s\Delta}$, Kannai's results imply that
\begin{equation}
\label{gtmack}
4\pi g_t(\cos \theta) \sim \frac{1}{s}e^{-\theta^2/4s}\sum_{j=0}^{\infty}v_j(\theta)s^j
\end{equation}
at least for $0 < |\theta| < \pi$, and where the $v_j$ are smooth functions.  (Kannai's result for general ${\bf M}$,
proved through use of the Hadamard parametrix for the wave equation, is that the kernel of $e^{-s\Delta}$ has the form
\[ \frac{1}{(4\pi s)^{n/2}} e^{-d(x,y)^2/4s}\sum_{j=0}^{\infty}V_j(x,y)s^j \]
for $x$ sufficiently close to $y$, where the $V_j$ are smooth.

In our case the geodesic distance from ${\bf N}$
to $z(\theta)$ is just $|\theta|$, and we have set $v_j(\theta) = V_j({\bf N},z(\theta))$.)

From (\ref{gtmack}) and the fact that (by (\ref{kersph5}) and (\ref{kersph6})), one has
\begin{equation}
\label{gthtrel}
\frac{\partial}{\partial s}g_t(\cos \theta) = -\frac{1}{s} h_t(\cos \theta),
\end{equation}
we are now motivated to find the first few terms in the
Maclaurin series (in $\theta$) of $e^{\theta^2/4s}4\pi g_t(\cos \theta)$ and $e^{\theta^2/4s}4\pi h_t(\cos \theta)$.
For this, we need only multiply the
right sides of (\ref{gtmac2}) and (\ref{htmac1}) by $$e^{\theta^2/4s} \sim \frac{1}{s}(s+\theta^2/4 + \cdots).$$ This yields
\begin{equation}\notag
e^{\theta^2/4s}4\pi g_t(\cos \theta) \sim
\frac{1}{s}[s\mbox{tr}(e^{-s\Delta}) + \frac{\theta^2}{4}[s\frac{d}{ds} \mbox{tr}(e^{-s\Delta}) + \mbox{tr}(e^{-s\Delta})]
\end{equation}
and
\begin{equation}
\label{hjloword}
e^{\theta^2/4s}4\pi h_t(\cos \theta) \sim
-\frac{1}{s}[s^2\frac{d}{ds}\mbox{tr}(e^{-s\Delta}) + \frac{\theta^2}{4}[s^2\frac{d^2}{ds^2} \mbox{tr}(e^{-s\Delta}) +
s\frac{d}{ds}\mbox{tr}(e^{-s\Delta})],
\end{equation}

where the error, for any fixed $s$, is $O(\theta^4)$. Combining this with Polterovich's result (\ref{heattrasx}), and
putting $s = t^2$ again, we obtain the approximations
\begin{equation}
\label{gtapp}
4\pi g_t(\cos \theta) \sim \frac{e^{-\theta^2/4s}}{s}[(1+\frac{s}{3}+\frac{s^2}{15}+\frac{4s^3}{315}+\frac{s^4}{315})
 + \frac{\theta^2}{4}(\frac{1}{3}+\frac{2s}{15}+\frac{4s^2}{105}+\frac{4s^3}{315})]
\end{equation}
It would take considerably more analysis to estimate the error here, but Maple says that when $t=.1$, the error
is never more than $6 \times 10^{-4}$ for any $\theta \in [-\pi,\pi]$, even though both sides have a maximum
of about 100.    Maple says that the greatest error occurs at around $\theta=.3$, where both sides are about
10.655.

Similarly one could use (\ref{hjloword}) to give an approximation to $4\pi h_t(\cos \theta)$.  But Maple says that
the errors are smaller if we differentiate formally with respect to $s$ in (\ref{gtapp}) (and then multiply by $-s$);
here we are recalling (\ref{gthtrel}).  If we do this and finally replace $s$ by $t^2$, this yields the approximation

\begin{equation}
\label{htapp}
4\pi h_t(\cos \theta) \sim \frac{e^{-\theta^2/4t^2}}{t^2}[(1-\frac{\theta^2}{4t^2})p(t,\theta)-t^2q(t,\theta)],
\end{equation}
where
\begin{equation}\notag
p(t,\theta)=1+\frac{t^2}{3}+\frac{t^4}{15}+\frac{4t^6}{315}+\frac{t^8}{315}+
\frac{\theta^2}{4}(\frac{1}{3}+\frac{2t^2}{15}+\frac{4t^4}{105}+\frac{4t^6}{315})
\end{equation}
and
\begin{equation}\notag
q(t,\theta)= \frac{1}{3}+\frac{2t^2}{15}+\frac{4t^4}{105}+\frac{4t^6}{315}+
\frac{\theta^2}{4}(\frac{2}{15}+\frac{8t^2}{105}+\frac{4t^4}{105})
\end{equation}

This approximation differs from the one obtained from (\ref{hjloword}) only in terms which are fourth order in $\theta$.
Maple says that when $t=.1$, the error in the approximation (\ref{htapp}) is never more than $9.5 \times 10^{-4}$
for any $\theta \in [-\pi,\pi]$, even though both sides have a maximum of about 100.  Maple says that the
greatest error occurs at around $\theta = .4$, where both sides are about -5.593.  Of course, if in (\ref{htapp})
we approximate $p \sim 1$ and $q \sim 0$, we would obtain the formula for the usual Mexican hat wavelet
on the real line, as a function of $\theta$.

Let us then explain how one can readily obtain any number of terms of the Maclaurin series of $g_t$ and $h_t$.

\begin{lemma}
\label{sphlpall}
For any positive integer $m$, there are constants $a_1,\ldots,a_m$ as follows.
If, in the situation of Lemma \ref{sphlpx}, $u$ is $C^{2m}$ on an
open interval containing $[-1,1]$, then
\begin{equation}\notag
\frac{d^{2m}}{d\theta^{2m}} [u(\cos \theta)]|_{\theta = 0} = \sum_{i=1}^m a_i (\Delta^i U)({\bf N}).
\end{equation}
Moreover, $a_m \neq 0$.
\end{lemma}
\begin{proof} It is enough to show that, generalizing (\ref{lapgx1}), there are polynomials
$p_1,\ldots,p_{2m}$ in $x_1$ such that
\begin{equation}
\label{lapgxall}
(\Delta^m U)(x) = \sum_{i=1}^{m} p_i(x_1) u^{(i)}(x_1) + \sum_{i=m+1}^{2m}(1-x_1^2)^{i-m}p_i(x_1) u^{(i)}(x_1)
\end{equation}
where $p_m(1) \neq 0$.  For then the mapping
 $$(u'(1),\cdots,u^{(m)}(1)) \mapsto
 \left((\Delta U)({\bf N}),\cdots,(\Delta^m U)({\bf N}) \right)$$
 will be given by an invertible upper triangular matrix.
In particular this is true if $n=1$, so that the mapping
$$\left(u'(1),\cdots,u^{(m)}(1)\right) \mapsto
\left(\frac{d^2}{d\theta^2} [u(\cos \theta)]|_{\theta = 0} ,\cdots,
\frac{d^{2m}}{d\theta^{2m}} [u(\cos \theta)]|_{\theta = 0} \right)$$
is also given by an invertible upper triangular
matrix.  Thus the map
$$
 \left((\Delta U)({\bf N}),\cdots,(\Delta^m U)({\bf N})\right) \mapsto
\left(\frac{d^2}{d\theta^2} [u(\cos \theta)]|_{\theta = 0} ,\cdots,\frac{d^{2m}}{d\theta^{2m}} [u(\cos \theta)]|_{\theta = 0} \right)$$
is also given by an invertible upper triangular matrix, which is the desired result.

To prove (\ref{lapgxall}), we recall first that by (\ref{lapgx1}), $(\Delta U)(x) = Du(x_1)$, where
$D = nx_1 \frac{d}{dx_1}-(1-x_1^2)\frac{d^2}{dx_1^2}$.  From this, (\ref{lapgxall}), save for the
statement that $p_m(1) \neq 0$, follows at once by a simple induction on $m$.
If $p_m(1)$ were zero, then the mapping
$$ \left(u'(1),\cdots,u^{(m)}(1) \right) \mapsto
 \left((\Delta U)({\bf N}),\cdots,(\Delta^m U)({\bf N})\right)$$
 would be given by a singular upper triangular matrix,
so its range would not be all of $\RR^m$.  But the range is all of $\RR^m$, since it contains
all the vectors
\[ \left((\Delta Z_l)({\bf N}),\ldots,(\Delta^m Z_l)({\bf N})) = Z_l({\bf N})l(l+n-1)(1,l(l+n-1),\ldots,[l(l+n-1)]^{m-1}\right), \]
for $l=1,\ldots,m$, and since the rows of a Vandermonde matrix are linearly independent.  This
contradiction completes the proof.\end{proof}

To obtain further terms of the Maclaurin series of $g_t$ or $h_t$, one need only use Lemma \ref{sphlpall}
together with the fact that applying $\Delta^i$ to the series in (\ref{kersph5}) or (\ref{kersph6})
is the same as applying $(-\partial /\partial s)^i$.  Then Polterovich's formula for the heat trace asymptotics may be
used.\\

\begin{remark}Throughout this discussion, we have been considering (\ref{kersph3}) for $f(s) = se^{-s}$.
If one instead used $f_m(s) = s^me^{-s}$, where $m$ is a nonnegative integer, one could still obtain explicit formulas,
similar to (\ref{htapp}), giving ``generalized'' Mexican hat wavelets on the sphere.)
Indeed, if $h_t^m$ corresponds to $f_m$ in the same way that $h_t$
corresponds to $f= f_1$, we have in place of (\ref{gthtrel}) that
$\frac{\partial^m}{\partial s^m}g_t(\cos \theta) = \frac{(-1)^m}{s^m} h_t^m(\cos \theta)$.
Using $f_m$ in place of $f$ has certain advantages.  In particular,
considering functions (of $\Delta$)
which vanish more quickly at $0$ (such as the $f_m$) is very important in the characterization
of Besov spaces in \cite{gm3}.\end{remark}

\section{Appendix: A Technical Lemma}\label{a-technical-lemma}
\label{cclem}

In the proof of Lemma \ref{manmol}, we used the following fact.
\begin{lemma}
\label{ccamplem}
Suppose $a_0,\ldots,a_L \in \CC$.  Then there exists an even function
$u \in {\mathcal S}(\RR)$, such that $u^{(2l)}(0) = a_l$, whenever $0 \leq l \leq L$,
and such that supp$\: \hat{u} \subseteq (-1,1)$.  (Of course, all {\em odd}
order derivatives of $u$ vanish at $0$.)
\end{lemma}
\begin{proof}
We construct $h = \hat{u}$.  We need to show that, given
numbers $b_0,\ldots,b_L$, there exists
$h \in C_c^{\infty}(\RR)$, $h$ even, supp$\:h \subseteq (-1,1)$, with
$\int_{-\infty}^{\infty} \xi^{2l} h(\xi) d\xi = b_l$
for $0 \leq l \leq L$.  Select an even function
$\eta \in C_c^{\infty}(\RR)$, with supp$\eta \subseteq (-1,1)$,
and $\int \eta = 1/2$.  Select $L+1$ different numbers $r_1,\ldots, r_{L+1} \in
(0,1/2)$.
For $0 < t < 1$, consider the even function
$\eta_{t,m}(\xi) = t^{-1}[\eta((x+r_m)/t) + \eta((x-r_m)/t)]$.  We claim that if
$t$ is sufficiently small, we may construct our $h$ by setting
$h = \sum_{m=1}^{L+1} c_m \eta_{t,m}$ for suitable $c_1,\ldots c_{L+1}$.  Indeed, as
$t \rightarrow 0^+$, $\eta_{t,m}$ becomes increasingly concentrated near $r_m$ and near $-r_m$,
so that $\int \xi^{2l} [\sum_{m=1}^{L+1} c_m \eta_{t,m}] d\xi \rightarrow
\sum_{m=1}^{L+1} c_m (r_m)^{2l}$.  The existence of suitable $c_1,\ldots c_{L+1}$ now follows
from the nonvanishing of small perturbations of the Vandermonde determinant. \end{proof}

\end{document}